\documentclass{amsart}

\title[Hubbard trees]{Combinatorial Hubbard trees for postcritically infinite unicritical polynomials and exponential maps}

\author{Malte Ha{\ss}ler}
\address{Cornell University, Department of Mathematics, Malott Hall, 212 Garden Avenue, 14853 Ithaca, USA}
\email{mh2479@cornell.edu}
\author{Dierk Schleicher}
\address{Aix-Marseille Universit\'e and CNRS, UMR 7373, Institut de Math\'ematiques de Marseille, 163 Avenue de Luminy Case 901, 13009 Marseille, France. 
}
\email{dierk.schleicher@univ-amu.fr}

\usepackage{amsthm}
\usepackage{graphicx}
\usepackage{enumitem}
\usepackage{float}

\usepackage{amsmath, amsthm, amsfonts,amssymb, comment}

\theoremstyle{plain}
\newtheorem{theorem}{Theorem}[section]
\newtheorem{lemma}[theorem]{Lemma}
\newtheorem{proposition}[theorem]{Proposition}
\newtheorem{corollary}[theorem]{Corollary}

\theoremstyle{definition}
\newtheorem{definition}[theorem]{Definition}

\newtheorem{example}{Example}
\newtheorem*{remark}{Remark}

\newcounter{reminder}

\newcommand{\Hide}[1]{}
\newcommand{\hide}[1]{}

\renewcommand{\theta}{\vartheta}
\newcommand{\0}{\mathtt{0}}
\newcommand{\1}{\mathtt{1}}
\newcommand{\2}{\mathtt{2}}

\newcommand{\vp}{\varphi}
\newcommand{\C}{\mathbb C}
\newcommand{\Z}{\mathbb Z}
\newcommand{\R}{\mathbb R}

\newcommand{\N}{{\mathbb N}}
\newcommand{\Hub}{\mathcal{H}}

\newcommand{\sm}{\setminus}

\newcommand{\ovl}[1]{\overline{#1}}

\newcommand{\diff}{\operatorname{diff}}

\newcommand{\ItSpec}{\1A_d^\infty}
\renewcommand{\phi}{\varphi}
\renewcommand{\varsigma}{\nu}

\newcommand{\Newpage}{}

\begin{document}

\begin{abstract}
We construct combinatorial Hubbard trees for all unicritical polynomials, and for all exponential maps, for which the critical (singular) value does not escape. More precisely, out of an external angle, or more generally a kneading sequence, we construct a forward invariant tree for which the critical orbit has the given kneading sequence. 

When the critical orbit of a unicritical polynomial is periodic or preperiodic, then the existence of a Hubbard tree is classical. Our trees exist even when the critical orbit is infinite (in many cases, this yields infinite trees), and even when the polynomial Julia set fails to be path connected. In the latter case, our trees cannot be reconstructed from the Julia set in complex dynamics. Our trees also exist for kneading sequences that do not arise in complex dynamics (for non complex-admissible kneading sequences). 

\end{abstract}

\maketitle

\section{Introduction}

Hubbard trees are an important concept to describe and distinguish postcritically finite polynomials. They are usually constructed as invariant subsets that connect all critical orbits within the filled-in Julia set  (subject to a regularity condition within bounded Fatou components), and are minimal with this property (that is, they are spanned by all postcritical points). 

For postcritically infinite polynomials, the filled-in Julia set is sometimes not path connected, and even if it is, to prove path connectivity is a highly non-trivial issue. We propose a combinatorial approach to construct a tree as a topological dynamical system for every unicritical polynomial (with connected Julia set), without any restriction on the dynamics, and without any knowledge on the topology of the Julia set. The resulting invariant tree satisfies all conditions of a Hubbard tree (suitably generalized to the postcritically infinite case), so if the Julia set admits such a tree, then it is topologically conjugate to the one we construct; and if the Julia set does not admit such a tree, then we view this as a problem of the Julia set that is completely irrelevant to our combinatorial construction. 

Our construction works for unicritical polynomials of all degrees, and even for exponential maps. It is known that even exponential maps with finite orbits of the singular value never have invariant Hubbard trees within the complex plane --- but once again this is a shortcoming of the complex dynamical plane that is irrelevant to our combinatorial construction. 

The existence of Hubbard trees is fundamental in the classification or combinatorial distinction of postcritically polynomials, and the same relevance should apply in the postcritically infinite case and for exponentials. Also, various dynamical invariants, such as core entropy, can easily be defined through Hubbard trees as soon as these exist. For instance, one consequence is that core entropy for unicritical polynomials of all degrees, and for all exponential maps, is bounded above by $\log 2$ \cite{MalteTransEntropy}. 

For every complex unicritical polynomials there is an external angle associated to the critical value (possibly finitely many angles), and in general the dynamics is uniquely determined in terms of only the degree and the external angle (subject to some subtleties such as local connectivity or rigidity of the parameter spaces that again is of no importance for us here). In our approach, every external angle has a unique associated kneading sequence (defined below) that completely describes the combinatorics of the polynomial. All our constructions are based solely on the kneading sequence, which is an infinite sequence over an alphabet with cardinality the degree of the polynomial, resp.\ infinity for exponentials. An interesting by-product is that our construction works for all kneading sequences, while not all kneading sequences arise in complex dynamics: that is, not all kneading sequences are \emph{complex admissible}. 

Our construction is, of course, related to much previous work in complex dynamics. Douady and Hubbard~\cite{Orsay} have constructed Hubbard trees for postcritically finite polynomials, and Poirier~\cite{Poirier} has used them to classify all of these in a dynamical way. Bruin, Kaffl, and Schleicher~\cite{BKS} have shown existence and uniqueness of abstract Hubbard trees for all periodic kneading sequences of degree $2$, whether or not they are complex admissible. Our work extends all these to the postcritically infinite case. Interestingly, our construction is much more straightforward than that in \cite{BKS} that treats only the (supposedly simple) periodic case of degree $2$ (which for us requires some more effort, due to the need of ``Fatou intervals''). Penrose~\cite{PenroseThesis} has constructed combinatorial models of dendrite Julia sets for all quadratic kneading sequences, complex admissible or not, but excludes the periodic case. From his model Julia set, we can in principle extract the Hubbard tree (as the closure of the connected hull of the critical orbit, but it is a separate step and had a separate goal). 

We should also mention Thurston's theory of invariant quadratic laminations \cite{ThurstonLami} that model Julia sets of quadratic polynomials, postcritically finite or not, based on a single combinatorial input, the external angle. Here, like for Penrose, we can extract the Hubbard tree in a separate combinatorial step. This construction works by design only for complex admissible kneading sequences.

Previous work on exponential Hubbard trees is limited. Pfrang, Rothgang, and Schleicher \cite{HomotopyTrees} have shown that even postsingularly finite exponentials never have invariant Hubbard trees, and have shown that the right concept from a combinatorial point of view are ``Homotopy Hubbard trees''. Our combinatorial construction does not suffer from the limitations of trees, embedded in $\C$, that arise for complex exponentials (or for certain non-locally connected Julia sets of polynomials).

Our fundamental approach is as follows: based on a kneading sequence $\nu$ (an infinite sequence on $\Z/d\Z$ for finite degrees $d\ge 2$, or on $\Z$ for degree $d=\infty$), we discuss unicritical maps with kneading sequence $\nu$: these are polynomials (or the exponential) for which there is a unique critical point such that the critical orbit has itinerary $\nu$ with respect to the partition of the (to-be-defined) Hubbard tree induced by the critical point. A \emph{precritical point of depth $n$} is a point on the Hubbard tree that, after $n$ iterations, maps to the critical value. We construct recursively all precritical points on an arc between the critical point on its image the critical value, based just on the kneading sequence $\nu$. The closure of this countable set is naturally homeomorphic to an interval, the \emph{critical path} in the Hubbard tree between critical point and critical value (for periodic kneading sequences, we have to insert certain ``Fatou intervals''). The Hubbard tree is then the union of the forward iterates of the critical path.

Our approach is completely symbolic. Every point on the Hubbard tree (and on the Julia set) can be described by an infinite sequence over a formal alphabet and the complex polynomial or exponential map becomes the left shift on sequences. 

Of course, all relevant propertes have to be shown: that the critical path is homeomorphic to an interval, that the forward union is uniquely path connected (a tree, possibly with infinitely many endpoints), that the dynamics given by the shift map is continuous with respect to an appropriate topology, and that the resulting tree has the expected properties. 

We should mention that exponential maps in $\C$, unlike polynomials, have no critical point or critical value, but an asymptotic value, known as the singular value, as the only point in $\C$ that has no neighborhood over which all branches of the inverse are locally biholomorphic. In our trees, even for $d=\infty$, there does exist a unique critical point (possibly with infinitely many branches at this point), so we can speak of critical point and critical value like for unicritical polynomials.

In Section~\ref{Sec:path} we construct the critical path from our symbolic set-up. In Section~\ref{Sec:tree} we define the Hubbard tree and investigate its dynamical structure. Finally, in Section~\ref{Sec:examples}, we give examples of Hubbard trees and conditions on their finiteness and recurrence. 

\Newpage

\section{Kneading sequences and the critical path}
\label{Sec:path}

We describe unicritical polynomials of degrees $d\ge 2$, as well as the exponential map (with $d=\infty$) in terms of symbolic dynamics. For every degree $d=2,3,\dots,\infty$, we have an alphabet $A_d=\{\0,\1,\2,\dots,\mathtt{d-1}\}$ resp.\ $A_\infty=\Z$ with $|A_d|=d$. To stress the analogy, especially as $d\to\infty$, we could also write $A_d=\{-\lfloor (d-1)/2\rfloor, \dots, -1,0,1,\lfloor d/2\rfloor\}$, but all that matters is a simple system of representatives. (The sets $A_d=\Z/d\Z$ come with a cyclic order, and $\Z$ comes with a linear order. This has an effect on embeddings of the tree into the plane, but not on our constructions which are compatible with arbitrary permutations of $A_d$ or $A_\infty$.)

Moreover we have the special symbol $\star$ to describe the critical point, and write $A_d^\star=A_d \cup \{\star \}$. The set of infinite sequences over the alphabet $A_d$ is described by $A_d^\infty$.

\subsection{Abstract kneading sequences}

\begin{definition}[Abstract kneading sequence]
An \emph{abstract kneading sequence} is one of the following:
\begin{itemize}
\item 
a non-periodic infinite sequence over the alphabet $A_d$, 
\item 
or an infinite periodic sequence over the alphabet $A_d^\star$ where the $\star$ symbol occurs exactly once within the period, at the last position. 
\end{itemize}
The latter kind we call a \emph{$\star$-periodic kneading sequence}. 
\end{definition}

By (historical) convention, every kneading sequence starts with the symbol $\1$ (with the exception of the \emph{trivial kneading sequence} $\ovl \star$, the unique sequence of period~$1$). This is no loss of generality because of the symmetry of $A_d$. We denote the subset of $A_d^\infty$ consisting of all sequences starting with the symbol $\1$ by  $\ItSpec$.

We denote periodic sequences by a bar above a finite sequence, such as $\ovl\star=\star\star\star\dots$ or $\ovl{\1\0}=\1\0\,\1\0\,\1\0\dots$. We also have preperiodic sequences such as $\1\ovl{\1\0}=\1\;\1\0\;\1\0\;\1\0...$.


\begin{definition}[Dynamical system associated to kneading sequence]
Every abstract kneading sequence $\nu$ ($\star$-periodic or not) has an associated dynamical system $(X_\nu,\sigma)$ where $X_\nu$ is a collection, described below, of infinite sequences over the alphabet $A_d
^\star$; the dynamics is given by $\sigma$, the left shift on sequences.  
\end{definition}

Intuitively, the collection $X_\nu$ models the Julia set of a unicritical polynomial (or an exponential) with kneading sequence $\nu$ (with certain necessary modifications when $\nu$ is $\star$-periodic because of the existence of bounded Fatou components). By definition, it contains the sequence $\nu$, called the \emph{critical value}, and the sequence $\star\nu$ called the \emph{critical point}, clearly with $\sigma(\star\nu)=\nu$. Moreover, it contains all \emph{postcritical points} $\nu_k:=\sigma^k(\star\nu)$ for $k\ge 1$, as well as all \emph{precritical points} of the form $w\star\nu$, where $w$ is an arbitrary finite, possibly empty word over $A_d$  (here and elsewhere, terms like $w\star\nu$ denote concatenation of the finite word $w$, followed by the symbol $\star$ and then the infnite sequence $\nu$).
 Finally, $X_\nu$ contains all sequences over $A_d$, except those of the form $w\nu$ where $w$ is a non-empty word over $A_d$: this way, we make sure that the critical point $\star\nu$ is the only $\sigma$-preimage of the critical value $\nu$.

We say that the \emph{depth} of a precritical point $w\star\nu$ equals $|w|+1$ (the depth is the number of iterations it takes the point to reach $\nu$). In particular, the depth of the critical point $\star\nu$ is $1$, and the depth of the critical value $\nu$ is $0$. 

We will often refer to the elements of $X_\nu$ as \emph{itineraries}.
A different point of view to construct the space $X_d$ is as a ``glueing space'' based on $A_d^\infty$ where all $d$ preimages of $\nu$ are identified (``glued'') into the single sequence $\star\nu$, and then requiring that the quotient be $\sigma$-invariant. These glueing spaces were investigated by Penrose for $d=2$ \cite{PenroseThesis}.

For $a,b\in X_\nu$, we define the value $\diff(a,b)\in\N\cup\{+\infty\}$ as the position of the first difference in the sequences $a$, $b$, where an entry $\star$ counts as ``wild card symbol'' that is not different from any symbol in $A_d$. 
We should note that $2^{-\diff(a,b)}$ is \emph{not} a metric on $X_\nu$ because it fails the triangle inequality at precritical points: for every precritical point $w\star\nu$ there are sequences in $X_\nu$ that are arbitrarily close to $w\star\nu$ but that are not close to each other: by replacing $\star$ with different symbols. 

Define cylinder sets $N_k(a)$ that consist of all sequences  $b\in X_\nu$ with $\diff(a,b)\ge k$. 
Let $\partial_k(a)$ denote the set of all precritical points in $N_k(a)\setminus \{a\}$ of depth up to $k-1$, except $\nu$. Note that these precritical points (but not $\nu$) are those that glue different cyclinders to each other so that the totally disconnected space $A_d^\infty$ becomes connected.

If $\nu$ is non-periodic, we define a topology on $X_\nu$ as follows. For $a\in X_\nu$ and $k\ge 1$ define the sets
\begin{equation}
U_k(a)=N_k(a)\sm \partial_k(a)
\label{Eq:Neighborhood_Uk}
\end{equation}
and declare these as a basis of open neighborhoods of $a$ in $X_\nu$. 

(It turns out that if $\nu$ is non-periodic then $X_\nu$ is path connected; this was shown in \cite{PenroseThesis} for $d=2$, and it follows similarly as in Lemma~\ref{Lem:CritPathInterval} below.)

If $\nu$ is $\star$-periodic, it is not true that $a,b \in X_\nu$ with $\diff(a,b)=\infty $ implies $ a=b$. This makes it difficult to define an adequate toplogy on $X_\nu$ (which is why Penrose had excluded this case). We will not define a topology on $X_\nu$ in this case either and overcome this issue by defining the critical path first and importing its topology to the Hubbard tree, which has the advantage of being finite in this case.

\begin{lemma}
\label{Lem:X_continuous}
For non-periodic $\nu$ the map $\sigma: X_\nu \mapsto X_\nu$ is continuous.
\end{lemma}
\begin{proof}

Choose $a\in X_\nu$. In the special case $a=\nu$, we clearly have  $\sigma^{-1}(U_k(\nu))=U_{k+1}(\star\nu)$ and continuity follows.

Otherwise, $a \neq \nu$ and there is a $K$ such that $\nu \not \in N_k(a)$ for all $k \ge K$. 
We claim that we have the following identity
\[
\sigma^{-1}(U_k(a))=\bigcup_{\mathtt e \in A_d} U_{k+1}(\mathtt e a)
\;,
\]
which directly implies continuity.

To see the identity, first observe that it clearly holds for $N_k(a)$ in place of $U_k(a)$, and for the boundary points it is easily checked. 
\end{proof}

\Newpage

\subsection{Precritical points on the critical path}

We are now ready to construct the \emph{critical path} of a kneading sequence $\nu$ except the trivial kneading sequence $\nu=\ovl\star$. From now on, when speaking of a kneading sequence, we will always implicitly exclude this degenerate case.

\begin{definition}[Precritical points on critical path]
\label{Def:PrecPointsCritPath}
Let $\nu$ be a non-trivial abstract kneading sequence, possibly $\star$-periodic. 
By induction, for every $n\ge 0$, we construct a finite, totally ordered set $P_n$ of precritical points and an injective, order-preserving mapping $\vp: P_n \mapsto [0,1]$  as follows. We set $P_1=\{\star\nu,\nu\}$ where we define the order $\star\nu\prec\nu$; we set $\vp(\star\nu)=0$ and $\vp(\nu)=1$. 

By induction, we construct $P_{n+1}$ from $P_n$ by inserting additional precritical points between adjacent elements with respect to the recursively defined order and so that the value of $\phi$ is compatible with this order. In particular, in all $P_n$ the largest element will always be $\nu$, the smallest will be $\star\nu$, and all other elements have the form $w\star\nu$ where $w$ is a finite sequence over $A_d$. 

To do this, let $a\prec b$ be two adjacent points in $P_n$. Assume first that $b\neq\nu$. Then we can write $a=w\star\nu$ and $b=w'\star\nu$ with finite sequences $w, w'$ (here $a$ could be the empty word). Let $k:=\diff(w\star\nu,w'\star\nu)$. 

If $k=\infty$, then no precritical point between $w\star\nu$ and $w'\star\nu$ is ever inserted.

If $k<\infty$, construct the unique precritical point $w''\star\nu$ with depth $k$ such that $w''$ is a string of length $k-1$ over $A_d$ that does not differ from the first $k-1$ entries in $w\star\nu$ and $w'\star\nu$ (if at a given position, both sequences do not have a $\star$, then they must agree, and this is the symbol for $w''$ at this position; if exactly one sequence has a $\star$, then $w''$ coincides with the other sequence at this position; and $k<\infty$ implies that adjacent precritical points never have a $\star$ at the same position). 

We extend the order to $w''\star\nu$ such that $w\star\nu\prec w''\star\nu\prec w'\star\nu$ and set $\vp(w''\star\nu)=(\vp(w\star\nu)+\vp(w'\star\nu))/2$. 

In the special case $b=\nu$, we still have $a=w\star\nu$ and can proceed as above, inserting a new sequence between $w\star\nu$ and $\nu$. 

Finally, the set $P_{n+1}$ consists of $P_n$ and all precritical points constructed between adjacent points in $P_n$. 
\end{definition}

Clearly, every set $P_{n+1}\setminus P_{n}$ is a finite set of precritical points of depth at least $n+1$. 
In general, most elements in $P_n$ will have depth greater than $n$ because in the construction $|w''|$ is at least $\max\{|w|,|w'|\}+1$ (if not, it would have been constructed earlier), but it can be much greater.

If in the definition above the case $k=\infty$ never occurs, that is if in every $P_n$ between any two adjacent precritical points another precritical point is inserted, then by induction, we have $|P_{n+1}|=2|P_n|-1$ for $n\ge 1$ and thus $|P_n|=2^{n-1}+1$. In every case $|P_n|\le 2^{n-1}+1$.
Since every precritical point of depth $n$ must be contained in $P_n$, the number of precritical points of depth $n$ is at most $2^{n-1}+1$. 

Our next task is to investigate when the case $k=\infty$ occurs. It turns out that this happens only when $\nu$ is $\star$-periodic, and only when it is a ``bifurcation'' in the following sense.


\begin{definition}[Bifurcating $\star$-periodic kneading sequences]
We say that a $\star$-periodic kneading sequence $\nu=\ovl{\nu_1\nu_2...\nu_{p-1}\star}$ of period $p$ is a \emph{bifurcation} from period $q$ if $q$ strictly divides $p$ and there is a symbol $e\in A_d$ such that $\ovl{\nu_1\nu_2...\nu_{p-1}e}$ has exact period $q$ (but not period lower than $q$).
\end{definition}

It is not hard to check that if such a $q$ exists, then it is uniquely determined by $\nu$: indeed, if there exists a different $q'$ with the same properties (and possibly a different last symbol $e'$), then for $q'':=\gcd(q,q')$ the first $p-1$ entries of $\nu$ must be repetitions of its first $q''$ entries. This in turn implies that $e'=e$, and for all symbols in $A_d\sm\{e\}$ at position $p$ the resulting sequence has exact period $p$. If $\nu$ is not a bifurcation, then all symbols at position $p$ lead to exactly period $p$.

If $\nu$ is a bifurcation from period $q$, then we call the $q$-periodic sequence $\mu=\ovl{\nu_1\nu_2\dots\nu_{q-1}\star}$ the \emph{base sequence of the bifurcation sequence $\nu$}. The base sequence $\mu$ can itself be a bifurcation from lower period.


Here are some examples of bifurcations: $\nu_1=\ovl{\1\1\0\,\1\1\star}$ and $\nu_2=\ovl{\1\1\0\,\1\1\0\,\1\1\star}$ are bifurcations from period $3$ with base sequence $\mu=\ovl{\1\1\star}$, while $\nu_3=\ovl{\1\1\1\,\1\1\star}$ is a bifurcation from period $1$. Moreover, $\ovl{\1\1\0\,\1\1\1\,\1\1\0\,\1\1\star}$ is a bifurcation from $\nu_1$. However,  $\ovl{\1\1\0\,\1\1\1\,\1\1\star}$ is not a bifurcation (but a so-called renormalization). 
While the examples of bifurcations given above (all for $d=2$) correspond to the familiar bifurcations in the Mandelbrot set, our definition also admits ``non-standard'' bifurcations to kneading sequences that are not realized in complex dynamics: for instance, $\ovl{\1\0\0\,\1\0\star}$ and $\ovl{\1\0\1\,1\0\star}$ are both bifurcation from base sequence $\ovl{\1\0\star}$, but the second one does not appear in the Mandelbrot set.
(It turns out that the case $k=\infty$ does not occur for those non-standard bifurcations; see Proposition~\ref{Prop:bifurcations}). 


\begin{lemma}[All precritical points different]
\label{Lem:no_gaps}
If $\nu$ is either non-periodic, or if it is $\star$-periodic but not a bifurcation, then in the construction of the precritical points in Definition~\ref{Def:PrecPointsCritPath} we always have $\diff(w\star\nu,w'\star\nu)<\infty$ (and also in the special case near the critical value, $\diff(w\star\nu,\nu)<\infty$). In particular, each $P_n$ with $n\ge 1$ contains exactly $2^{n-1}+1$ mutually distinct elements. 
\end{lemma}
\begin{proof}
We compare two adjacent precritical points such as $w'\star\nu$ and $w\star\nu$. If $w'\star\nu$ and $w\star\nu$ differ among their first $\max\{|w'|,|w|\}$ entries, then there is nothing to prove (but this case does not arise in the recursive construction of the critical path). Otherwise, the comparison boils down to comparing $\nu=\nu_1\nu_2...$ with $\sigma^t{\nu}=\nu_{t+1}\nu_{t+2}...$ for $t=\left||w'|-|w|\rule{0pt}{10pt}\right|$.
We call $t$ the shift between the precritical points. If $\nu$ is non-periodic, clearly $\diff(\nu, \sigma^t(\nu))<\infty$ and hence there is always a new precritical point constructed between two adjacent precritical points. 

It thus remains to consider the case that $\nu$ is $\star$-periodic, say with period $p$. Our main claim is that $t$ is not a multiple of $p$ for all pairs of adjacent points.  Using this claim, $\diff(\nu, \sigma^t(\nu))=\infty$ implies that $\nu$ is a bifurcation from some period dividing $\gcd(t,p)<p$; this completes the proof of the lemma. 

We prove the claim using induction on all precritical points of $P_n$. For $P_1$ we have $t=1<p$. Now suppose $w\star\nu$ and $w'\star\nu$ are adjacent precritical points of $P_n$ with shift $t$ not a multiple of $p$ and between them we construct $w''\star\nu$. 

 At the first entry where $w\star\nu$ and $w'\star\nu$ differ, we compare the entries $\nu_{s}$ and $\nu_{t+s}$ for some $s\ge1$. We note for later use that \emph{the entries $\nu_s$ and $\nu_{t+s}$ are different from each other and from $\star$}. 
By construction, this position corresponds to the first $\star$ in the sequence $w''\star\nu$. So comparing $w''\star\nu$ with $w\star\nu$ boils down to comparing $\nu$ with $\sigma^s(\nu)$ or $\sigma^{t+s}(\nu)$ (depending on which of $w$ and $w'$ was shorter). In other words, if we denote the shift between $w''\star\nu$ and $w\star\nu$ by $t'$,  we have $t'\in \{s, t+s\}$. 

If $t'$ is a multiple of $p$, then $\nu_{t'}=\star$ and thus $\nu_s=\star$ or $\nu_{t+s}=\star$, but we had noted above that this is not the case.

The same argument holds for the shift of $w'\star\nu$ and $w''\star\nu$.
\end{proof}

\subsection{Limit points and Fatou intervals}
\label{Sub:LimitsFatouIntervals}

Our next goal is to extend the (countable or possibly finite) set $P_\infty:=\bigcup_{n\ge 1} P_n$ with the mapping $\phi\colon P_\infty\to[0,1]$ to an interval, the \emph{critical path} $C_\nu$, so that $\phi$ extends to a homeomorphism from $C_\nu$ to $[0,1]$. Using the order topology on domain and range, what we will construct is an order preserving bijection.

In the case of Lemma~\ref{Lem:no_gaps}, that is when the case $\diff(w\star\nu,w'\star\nu)=\infty$ does not occur, the set $P_\infty$ is a countable set of precritical points; if  $\nu$ is not $\star$-periodic, its closure is naturally an interval. We will shows this in Lemma~\ref{Lem:limitP} below.

However, for $\star$-periodic kneading sequences, taking the closure turns out to be more complicated for two reasons. If in the construction of $P_n$ in  Definition~\ref{Def:PrecPointsCritPath}, the case $\diff(w\star\nu, w'\star\nu)=\infty$ occurs, then we have $\phi(w'\star\nu)>\phi(w\star\nu)$ but no further precritical points are ever inserted between them. We close such gaps by inserting abstract \emph{Fatou intervals} between the boundary points. This problem occurs for $\star$-periodic sequences that are bifurcations. Moreover, for arbitrary $\star$-periodic sequences $\nu$ there may be  itineraries $a,b\in X_\nu$ with $\diff(a,b)=\infty$ that have the same $\phi$-values, so $\phi$ would fail to be injective. We solve this problem also by inserting \emph{Fatou intervals} between such points.

\medskip

The first step in the construction of the critical path is to construct limits of points in $P_\infty$, similar to the usual completion such as when extending the rational numbers (or rather the dyadic numbers) to the reals, but with subtle differences. We identify a set of limits $L_\nu\subset X_\nu\cap A_d^\infty$ as follows. We say that $a\in X_\nu\cap A_d^\infty$ is in $L_\nu$ if there exists a sequence $(a_n)_{n\ge 0}\subset P_\infty$ with $a_n\in P_n$ for all $n\ge 1$ and $a_0\in P_1$, and with the property that for all $n$ we have either $a_{n} \prec a_{n+2} \prec a_{n+1}$ or $a_{n+1} \prec a_{n+2} \prec a_{n}$, and such that $\lim_{n\to \infty} \diff(a_n,a)=\infty$. 
For each limit $a\in L_\nu$ the sequence $(a_n)$ that converges to $a$ is unique because of the condition $a_n\in P_n$. 

As constructed, $L_\nu\cap P_\infty=\emptyset$, but certain points in $L_\nu$ duplicate existing points in $P_\infty$. We define $M_\nu \subset L_\nu$ as the set of limit points $a \in L_\nu$ for which there exists a precritical point $w\star\nu \in P_\infty$ such that $\diff(a,w\star\nu)=\infty$ ($a\in A_d^\infty$ implies $a\neq w\star\nu$). Note that if $\nu$ is non-periodic then $X_\nu$ is Hausdorff so $M_\nu=\emptyset$. However in the $\star$-periodic case, in the sequence $w\star\nu$ the symbol $\star$ can be replaced by different elements in $A_d$, and up to two of these may be in $L_\nu$ (the point $w\star\nu$ can be approximated from the left or from the right with respect to the order $\prec$). 

The critical path will be the set $P_\infty\cup L_\nu$, extended by Fatou intervals as needed.

\begin{lemma}[Limit points of $P_\infty$]
\label{Lem:limitP}
If $\nu$ is not a bifurcation, then $P_\infty\cup (L_\nu \setminus M_\nu)$ is homeomorphic to $[0,1]$. More precisely, the map $\phi\colon P_\infty\to[0,1]$ extends to an order preserving bijection. 

If $\nu$ is a bifurcation, then $\phi$ still extends to an order preserving bijection from $P_\infty\cup (L_\nu \setminus M_\nu)$ to a closed subset of $[0,1]$, where gaps in the range correspond to precritical points $w\star\nu$ and $w'\star\nu$ with $\diff(w\star\nu,w'\star\nu)=\infty$ in Definition~\ref{Def:PrecPointsCritPath}, such that $\phi(w\star\nu)$ and $\phi(w'\star\nu)$ are different dyadic numbers in $[0,1]$.
\end{lemma}
\begin{proof}
Since there are only finitely many elements in $P_\infty$ with a $\star$ among the first fixed number of entries, any limit of a sequence in $P_\infty$ must be in $A_d^\infty$, and the order on $P_\infty$ extends to all the limit points. The construction implies that $|\phi(a_n)-\phi(a_{n+1})|=2^{-n}$, so $\phi(a_n)$ is a Cauchy sequence and we can set $\phi(a):=\lim_{n\to\infty} \phi(a_n) \in[0,1]$. 
Since every $a\in L_\nu\sm M_\nu$ is approximated by a unique sequence, the value $\phi(a)$ does not depend on any choices.

Note that $\phi(a) \in \phi(P_\infty)$ if and only if $a\in M_\nu$. So, by extension we have a strictly monotone map $\phi\colon (P_\infty\cup (L_\nu\setminus M_\nu))\to[0,1]$. 

If $\nu$ is not a bifurcation, the values $\phi(P_n)$ contain all numbers $m/2^n$ for $m\in\{0,1,\dots,2^n\}$, so $\phi(P_\infty)$ is dense in $[0,1]$. It also follows that $\phi\colon (P_\infty\cup (L_\nu \setminus M_\nu))\to[0,1]$ is surjective. 

For injectivity, suppose there are $a,b \in P_\infty\cup (L_\nu\setminus M_\nu)$ with $\phi(a)= \phi(b)$. Let $a_n, b_n\in P_n$ elements of the sequences that approximate $a$ resp.\ $b$ (if $b\in P_\infty$, we allow that $b_n$ is the constant sequence). Since $\phi(a_n)$ and $\phi(b_n)$ have the same limit, we have $\lim_{n\to \infty} \diff(a_n,b_n)=\infty$, so $\diff(a,b)=\infty$. Since $a\not \in M_\nu$, this implies $a=b$. 
Therefore, $\phi\colon (P_\infty\cup (L_\nu\setminus M_\nu))\to[0,1]$ is a bijection and thus an order preserving homeomorphism with respect to the order topologies on domain and range. 

If $\nu$ is a bifurcation, then $\phi$ has gaps exactly at precritical points $w\star\nu$ and $w'\star\nu$ with $\diff(w\star\nu,w'\star\nu)=\infty$ so that in Definition~\ref{Def:PrecPointsCritPath} no precritical points are ever inserted between $w\star\nu$ and $w'\star\nu$.
\end{proof}


If $\nu$ is $\star$-periodic but not a bifurcation (for instance $\nu=\ovl{\1\0\star}$), then the homeomorphism from $P_\infty\cup(L_\nu\sm M_\nu)$ corresponds, in the actual Hubbard tree of a quadratic polynomial, to the path connecting the critical point to the critical value, but collapsing each interval through a bounded Fatou component to a point. The resulting interval does not change its topological type, and the same is true for the resulting Hubbard tree. But this is not in the spirit of our construction.

In order to deal with gaps in $P_\infty$ and limit points in $M_\nu$  that occur for certain $\star$-periodic sequences, we introduce \emph{Fatou intervals}.

\begin{definition}[Fatou intervals]
Let $\nu$ be a $\star$-periodic kneading sequence $\nu$. For a precritical point $w\star\nu$ (possibly $\nu$ itself) and a sequence $w'' \in X_\nu \cap A_d^\infty$ such that $\diff(w\star\nu, w'')=\infty$, we define the \emph{Fatou interval} 
\[ 
[w'', w\star\nu ]
\] 
as abstract interval homeomorphic to $[0,1]$ that is disjoint from $X_\nu$ and $P_\infty$,  except that the endpoints are sequences in $P_\infty$ resp.\ $X_\nu$. 

We define the $\sigma$ map on each of these intervals as a homeomorphism to the Fatou interval $[\sigma(w''), \sigma(w\star\nu)]$.
\end{definition}

Recall that for a $\star$-periodic kneading sequence $\nu$ we have an order-preserving homeomorphism  $\phi\colon \, P_\infty\cup (L_\nu \setminus M_\nu) \mapsto C$, where $C$ is a closed subset of $[0,1]$. In a first step we extend this to appropriate Fatou intervals so that we can extend the definition of $\phi$ to a surjective map. In a second step we have to deal with the points in $M_\nu$ by inserting further Fatou intervals. 

Recall that a gap occurs when there are two precritical points $w\star\nu$ and $w'\star\nu$ in Definition~\ref{Def:PrecPointsCritPath} with $\diff(w\star\nu,w'\star\nu)=\infty$ and $\phi(w\star\nu)<\phi(w'\star\nu)$ so that no precritical point is ever inserted between these two points, leaving a gap between their $\phi$-images. 

There clearly exists a (pre)periodic itinerary $w'' \in X_\nu \cap A_d^\infty$ (without $\star$) with $w\star\nu\prec w''\prec w'\star\nu$ such that $\diff(w\star\nu, w'')=\diff(w'',w'\star\nu)=\infty$. This sequence is unique because the shift between $w\star\nu$ and $w'\star\nu$ cannot be a multiple of the period of $\nu$ (see the proof of Lemma~\ref{Lem:no_gaps}). We call the sequence $w''$ \emph{the central itinerary in the gap between $w\star\nu$ and $w'\star\nu$}. We close the gap by inserting the two Fatou intervals $[w\star\nu, w'']$ and $[w'', w'\star\nu]$ in that order. Extend the map $\phi$ to these Fatou intervals such that $\phi(w''):=(\phi(w\star\nu)+\phi(w'\star\nu))/2$ and interpolating $\phi$ along these intervals. 

The extended map $\phi$ from $P_\infty\cup (L_\nu\sm M_\nu)$ union all Fatou intervals is now monotone and surjective onto $[0,1]$. There is a unique monotone way to also extend it to $M_\nu$ so that we obtain a monotone surjective map $\phi\colon P_\infty\cup L_\nu\to [0,1]$, but it fails to be injective at $M_\nu\subset L_\nu$. More precisely, there are ``adjacent'' points $w\star\nu\in P_\infty$ and $a\in M_\nu$ with $\diff(w\star\nu,a)=\infty$ and thus with equal values of $\phi$ (the order between $w\star\nu$ and $a$ can be either way). The solution is to to insert additional Fatou intervals between $w\star\nu$ and $a$ as above, while shifting the values of $\phi$ on either side so as to ``make room'' for the values required in the new Fatou interval (we lose the dyadic values of $\phi$ that we had defined earlier). The set $P_\infty$ is countable, so we need to insert at most countably additional Fatou intervals in this step, and if we make them short enough (as measured by the depth of the precritical point belonging to a Fatou interval), the total length is finite. Rescaling, we obtain an order preserving homeomorphism $\tilde\phi$ from $P_\infty\cup (L_\nu\sm M_\nu)$ union all Fatou intervals added in both steps to $[0,1]$ 
as required.

\subsection{Illustrating examples}
Let us illustrate Fatou intervals with a few examples in connection to Julia sets of complex admissible polynomials. Their critical paths are illustrated in Figure~\ref{Fig:crit_paths} and their Hubbard trees as defined later are shown in Figure~\ref{Fig:h_trees}. If a kneading sequence is not $\star$-periodic, for instance if it is strictly preperiodic, then $X_\nu$ is a dendrite (as shown in \cite{PenroseThesis} for $d=2$), and no Fatou intervals are necessary, corresponding to the fact that the filled-in Julia set does not contain bounded Fatou components. 

Let us now consider the $\star$-periodic sequence $\nu=\ovl{\1\star}$ of period $2$ corresponding to the unique quadratic polynomial with superattracting $2$-cycle. Here we only have two precritical points $\nu=\ovl{1\star}$ and $\star\nu=\ovl{\star\1}$: the critical value and the critical point, and together they form the entire critical orbit. There are no further precritical points in $P_\infty$, indeed $P_\infty=P_1$, and the two points $\nu$ and $\star\nu$ form a single gap. The central itinerary in this gap is $w''=\ovl{\1}=:\alpha$ (the $\alpha$ fixed point), and we introduce two Fatou intervals $[\nu,\alpha]$ and $[\alpha,\star\nu]$. Together they form the complete critical path (and in fact the complete Hubbard tree as defined below).

A bit more interesting is the period $4$ sequence $\nu=\ovl{\1\0\1\star}$ that bifurcates from the period $2$ sequence just discussed. Here $P_2$ consists of the three postcritical points $\star\nu, \nu$ and $\sigma^2(\nu)=1\star\nu$ and between the latter two, two Fatou intervals are constructed, which are separated by the central itinerary $\overline{\1\0}$. (Together with its image they form the ``little $\alpha$ fixed point'' of the $2$-renormalization of the Julia set). The set $P_\infty$ is infinite. ) 

These were two examples of $\star$-periodic kneading sequences that are bifurcations and require gaps to be filled by Fatou intervals. An example of a different kind is the $\star$-periodic kneading sequences $\nu=\ovl{\1\0\star}$ of period $3$ that is not a bifurcation (it is realized by the real ``airplane'' polynomial of period $3$). In this case, there are no gaps to fill because we have shown that $\phi(P_\infty)$ is already dense in $[0.1]$. However, there is the limit sequence $\ovl{\1\0\1}\in M_\nu$ with $\diff(\ovl{\1\0\1},\nu)=\infty$. A priori we could identify these two points and the analogous points on the backward orbit and this would yield a viable Hubbard tree: the same is true for the actual Hubbard tree of this polynomial with a real superattracting $3$-cycle because this tree contains countably many intervals in bounded Fatou components, all with disjoint closures, and collapsing these would yield a topologically conjugate tree (unlike the previous two examples where the tree would shrink to a point). Our approach is to keep the two points $\nu$ and $\ovl{\1\0\1}$ distinct and to connect them by another Fatou interval. This will avoid inconsistencies for example with the sequence described next.

A possibly surprising example is the bifurcation from $\nu=\ovl{\1\0\star}$ to $\nu'=\ovl{\1\0\1\,\1\0\star}$: the first is again the real ``airplane'' polynomial of period $3$, while $\nu'$ does not occur in the Mandelbrot set and is thus not complex admissible. Here $\phi(P_\infty)$ is still dense and thus central itineraries associated to gaps do not occur as for the complex admissible bifurcations. Once again, the limit sequence $\ovl{\1\0\1}\in L_\nu$ has $\diff(\ovl{\1\0\1,\nu})=\infty$ and we insert a Fatou interval between these. Here the sequence $\nu$ has period $6$ while the limit sequence only has period $3$, so $\ovl{\1\0\1}$ must be a branch point in the Hubbard tree that we will construct later. If we had identified $\nu$ with $\ovl{\1\0\1}$, then this would imply an identification also of $\nu$ with $\sigma^3(\nu)$ and the Hubbard tree of period $6$ would collapse to that of the period $3$ sequence $\ovl{\1\0\star}$.

In Proposition~\ref{Prop:bifurcations}, we will provide a full classification of $\star$-periodic kneading sequences and further explain the differences observed in these examples.

\subsection{The critical path and its properties}

\begin{definition}[The critical path]
For every kneading sequence $\nu$ we define the critical path as $C_\nu:=P_\infty\cup L_\nu$ union all Fatou intervals.
\end{definition}

In all cases, the critical path comes with a map $\tilde \phi\colon C_\nu\to[0,1]$ (that coincides with the original map $\phi\colon P_\infty\to[0,1]$ except when $M_\nu\neq\emptyset$). 

\begin{lemma}[Critical path is interval]
\label{Lem:CritPathInterval}
For every $\nu$, the critical path is homeomorphic to $[0,1]$, with an homeomorphism given by $\tilde\phi\colon P_\infty\cup L_\nu \mapsto [0,1]$, extended to all Fatou intervals.
\end{lemma}
\begin{proof}
In the non-periodic case, this was shown in Lemma~\ref{Lem:limitP} because $M_\nu=\emptyset$. For bifurcations, we need to extend the original map $\phi$ to Fatou intervals as explained in Section~\ref{Sub:LimitsFatouIntervals}. The resulting map $\tilde\phi$ is an order preserving bijection, hence a homeomorphism with respect to the order topology on domain and range.
\end{proof}

Let us discuss the critical path starting with an elementary property.

\begin{lemma}[$\alpha$ fixed point]
\label{Lem:alpha}
The fixed point $\alpha:=\ovl\1$ is always an element of the critical path $C_\nu$.
\end{lemma}
\begin{proof}
If $\nu$ does not contain an entry from  $A_d\setminus \{\1\}$, then it must have the form $\nu=\ovl{\1\1\dots\1\star}$ and the set $P_\infty$ equals $P_1=\{\star\nu,\nu\}$. In this case $C_\nu$ contains $\alpha$ as the central itinerary associated to the gap between $\star\nu$ and $\nu$. 

Otherwise, we construct a sequence of precritical points starting with $\rho_0:=\nu$ and $\rho_1:=\star\nu$. By induction, let $\rho_{n+1}$ be the precritical point constructed first between $\rho_n$ and $\rho_{n-1}$, the points exist by our assumption on $\nu$. It is easy to see by induction that the itinerary of each $\rho_n$ starts with at least $n-1$ consecutive entries $\1$. Therefore, this sequence of precritical points converges to the itinerary $\ovl\1\in C_\nu$. 
\end{proof}

Since the critical path is homeomorphic to $[0,1]$, the open intervals $(a,b)=\{ \zeta \in C_\nu: a \prec \zeta \prec b \}$, as well as $[\star\nu, a)$ and $(b,\nu]$ form a basis for the topology.

If $\nu$ is non-periodic, hence in $A_d^\infty$, then we have defined a topology on $X_\nu$, and $C_\nu$ inherits a topology as a subset of $X_\nu$. These topologies coincide:

\begin{lemma}[Equivalent topologies]
\label{Lem:path_top}
If $\nu$ is non-periodic, then the topology of $C_\nu$ is equivalent to the topology it inherits as a subset of $X_\nu$. 
\end{lemma}
\begin{proof}
Fix $c\in C_\nu$ and let $U$ be an open interval in $C_\nu$ with respect to the topology of $C_\nu$. Our first observation is that if $c \in [a,b]$, then $\diff(a,c)\ge \diff(a,b)$. This holds by construction for elements of $P_\infty$ and the property extends to limits. Let $a$ and $b$ be the boundary points of $U$ and let $k=\max(\diff(a,c),\diff(b,c))+1$. Using the sets $U_k$ as defined in \eqref{Eq:Neighborhood_Uk}, it is clear that $U_k(c)\cap C_\nu \subset N_k(c)\cap C_\nu \subset U$, so $U$ is open in the subspace topology of $X_\nu$. 

For the converse, we show that for every basis set $U_k(c)\subset X_\nu$ (with $k\ge 1$) the intersection $U_k\cap C_\nu$ contains an open interval containing $c$. It follows from our first observation above that the set $N_k(c)\cap C_\nu$ is connected, so it is an interval of which $c$ is not an endpoint (except when $c$ is an endpoint of $C_\nu$). 
\end{proof}

\begin{figure}[htbp]
\includegraphics[width=.67\textwidth]{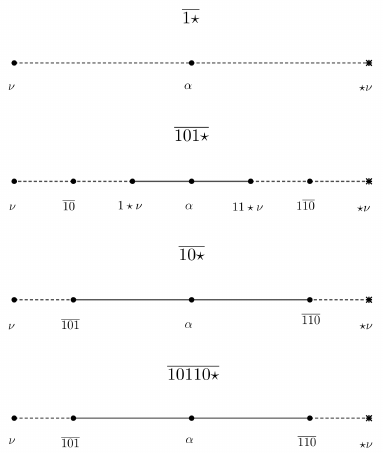}
\caption{The critical paths of four $\star$-periodic kneading sequences. Prominent Fatou intervals are indicated by dotted lines.}
\label{Fig:crit_paths}
\end{figure}

\Newpage

\section{The abstract Hubbard tree}
\label{Sec:tree}

\subsection{Constructing the Hubbard tree}

We will construct the Hubbard tree associated to a $\star$-periodic or non-periodic kneading sequence $\nu$ as union of (finitely or infinitely many) $\sigma$-iterates of the critical interval $C_\nu$. 

If, for $n\in\N$, the map $\sigma^n$ is injective on $C_\nu$, then the order topology on $C_\nu$ easily translates to $\sigma^n(C_\nu)$ so that $\sigma^n(C_\nu)$ is homeomorphic to $[0,1]$ as well. If this is not the case, we have to divide $C_\nu$ into finitely many sub-intervals on which $\sigma^n$ is injective.

For each $n\ge 2$, let $N(n)$ be the number of precritical points in $C_\nu$ of depth at most $n$. Denote these points by $\rho_i^n$ where $\star\nu \prec \rho_1^n \prec ... \prec \rho_{N(n)}^n\prec \nu$. Then each of the $N(n)+1$ intervals $[\star\nu, \rho_1^n], [\rho_1^n,\rho_2^n], ... , [ \rho_{N(n)}^n, \nu ]$ are mapped injectively under $\sigma^n$. Hence we have the order topology on the image. Thus each element in the finite set
\[
\mathcal I_n(\nu)=\left\{ \sigma^n([\star\nu, \rho_1^n]), \sigma^n([\rho_1^n,\rho_2^n]), ... , \sigma^n([ \rho_{N(n)}^n, \nu ])\right\}
\]
is homeomorphic to $[0,1]$. The first two cases are $\mathcal I_0(\nu)=\{[\star\nu,\nu]\}$ and $\mathcal I_1(\nu)=\{ \sigma([\star\nu,\nu])\}$. Note that the union of all elements of $\mathcal I_n(\nu)$ equals $\sigma^n(C_\nu)$. 

\begin{definition}[Finite trees $H_n$]
Define the sets 
\begin{equation*}
H_n:=\bigcup_{k=0}^{n}\sigma^k(C_\nu)
\;.
\end{equation*}
A subset $U \subset H_n$ is open if for all $k \in \{0,...,n\}$ and all intervals $J \in \mathcal I_k(\nu)$, the set $U\cap J$ is open in the topology of $J$. 
\end{definition}

In other words, the topology on $H_n$ is the quotient topology based on the intervals in $\mathcal I_n(\nu)$ with respect to the natural quotient given by identification of points with equal itineraries.

For a topological space $T$ we say that $P\subseteq T$ is a \emph{path} if $P$ is homeomorphic to $[0,1]$. A topological space $T$ is a \emph{tree} if it is a countable union of paths such that for every distinct $a,b \in T$ there exists a unique path $P$ from $a$ to $b$; this path is denoted $[a,b]$ (these paths are not oriented: $[a,b]=[b,a]$).  A point $x$ in a tree $T$ is called an \emph{endpoint} if $T\sm\{x\}$ is connected.


\goodbreak

\begin{lemma}[The finite trees $H_n$]
\label{Lem:Hn}
The sets $H_n(\nu)$ satisfy the following properties:
\begin{enumerate}
\item \label{it:tree} $H_n$ is a finite tree;
\item \label{it:endp} $H_n$ contains the set $\{\star\nu, \nu, ... \sigma^n\nu\}$, and all endpoints of $H_n$ are in this set; 
\item \label{it:inj} if a connected subset of $H_n$ does not contain the critical point, then all itineraries in this subset start with the same symbol in $A_d$;
\item \label{it:dense} 
for each itinerary $a\in H_n$, and every component $K$ of $H_n\setminus \{a\}$ there is a sequence of precritical points in $K$ that converges to $a$, except if $a$ is on the boundary of a Fatou interval in $K$. In particular, precritical points are dense in $H_n$ if $\nu$ is non-periodic. 
\end{enumerate}
\end{lemma}

\begin{proof}
We prove this by induction. For $n=0$, the set $H_0=[\star\nu,\nu]=C_\nu$ is homeomorphic to $[0,1]$ and thus a tree that satisfies all given properties.

Now suppose $H_n$ satisfies all listed properties; we will show the same for $H_{n+1}=H_n \cup \sigma(\sigma^n(C_\nu))$. 

Set $S_n:=[\sigma^n \nu,\sigma^{n-1} \nu]\subset H_n$. 
We claim that $H_{n+1}=H_n \cup \sigma(S_n)$. This is obvious for $n=0$, interpreting $\sigma^{-1}(\nu)$ as $\star\nu$.

For $n\ge 1$, observe that $\sigma^n(C_\nu)\setminus S_n\subset H_{n-1}$ because $\sigma^n(C_\nu)\setminus S_n\subset H_n$ by construction and all its endpoints are in $H_{n-1}$, so we can apply \eqref{it:endp}. Therefore $\sigma(\sigma^n(C_n))\subset H_n\cup \sigma([\sigma^n\nu,\sigma^{n-1}\nu])$ and the claim follows.

Now we prove properties \eqref{it:tree}, \eqref{it:endp}, and \eqref{it:inj} distinguishing the two cases whether or not  $\star\nu \not \in S_n$.

If $\star\nu \not \in S_n$, the main step is to establish that $\sigma(S_n)$ is a path in the topology of $H_{n+1}$. We know that $S_n\subset H_n$ is homeomorphic to $[0,1]$ by inductive hypothesis. By inductive assumption~\eqref{it:inj} all entries in $S_n$ start with the same symbol, so it maps injectively by $\sigma$. Since $H_n$ and $H_{n+1}$ are compact as finite unions of sets homeormorphic to $[0,1]$, it suffices to show that $\sigma\colon S_n\to\sigma(S_n)$ is continuous; then this is a homeomorphism and $\sigma(S_n)$ is a path.

Let $U$ be an open subset of $\sigma(S_n)$ in $H_{n+1}$, i.e. $U\cap J'$ is open for all $J' \in \mathcal I_{n+1}$. 

For each $J\in \mathcal I_n$, the set $J\cap  S_n$ does not contain the critical point, so their $\sigma^n$-preimages do not contain  precritical points of depth up to $n+1$. Hence, we have a one-to-one correspondence between those $J \in \mathcal I_n$ for which $J\cap  S_n$ is non-empty and those $J'\in \mathcal I_{n+1}$ for which $J'\cap\sigma( S_n)$ is non-empty. Moreover, the map $\sigma$ from $J\cap  S_n$ to the corresponding $J'\cap\sigma( S_n)$ is an order-preserving homeomorphism. Thus, the preimage of $U\cap J'$ in $J$ is open for the interval $J$ corresponding to $J'$. And thus, the preimage of $U\cap J'$ is open in all other elements of $\mathcal I_n$ as well. We conclude that $\sigma$ is continuous.

Therefore, $\sigma(S_n)$ is indeed a path from $\sigma^{n+1}\nu$ to $\sigma^n\nu$, but at this point, it is not clear that it is the only path; but we can define anyway $S_{n+1}:=[\sigma^n\nu, \sigma^{n+1}\nu]:=\sigma(S_n)$. 

Define $b=\sup\{ x \in  S_{n+1}\colon\, x \in H_n\}$. Since paths and finite trees are closed, $b$ itself lies in the intersection $H_n \cap S_{n+1}$. Let $a\in S_n$ be the unique preimage of $b$. Then $[\sigma^{n-1}\nu,a]\in H_{n-1}$, so its $\sigma$-image is the path $[\sigma^n\nu,b]\subset H_n$. At the same time, $\sigma([\sigma^{n-1}\nu,a])$ equals the subpath from $\sigma^n\nu$ to $b$ in $S_{n+1}$. Therefore, all points in $S_{n+1}$ between $\sigma^n\nu$ and $b$ are in $H_n$, while the remaining points in $S_{n+1}$ are not in $H_n$. Therefore, $H_{n+1}$ is a finite tree that extends $H_n$ by one additional path $[b,\sigma^{n+1}\nu]$. This completes the proof of properties \eqref{it:tree} and \eqref{it:endp}.

The new branch $B:=(b, \sigma^{n+1}\nu]$ does not contain the critical point or the points $\nu$, $\sigma\nu$, \dots, $\sigma^n \nu$ because all these points are already in $H_n$. Therefore $B$ is the $\sigma^{n+1}$-image of a subinterval of the critical path, say $B_C$, that does not contain precritical points of depth in $\{1,\dots,n+2\}$. By construction of the critical path, all itineraries on $B_C$ share the first $n+2$ entries, so all itineraries in $B$ start with the same symbol in $A_d$, proving  \eqref{it:inj}.

If $\star\nu \in S_n$, we split up $S_n$ into $[\sigma^{n-1}\nu, \star\nu] $ and $[\star\nu, \sigma^n\nu]$. The image of the first sub-path lies in $H_{n}$, so there is nothing to show and $H_{n+1}=H_n\cup\sigma([\star\nu, \sigma^n\nu])$. For the second sub-path we can proceed just as in the first case; the fact that $\star\nu$ is now a boundary point of an interval does not affect injectivity of $\sigma$.

Regarding $\eqref{it:dense}$, since the property holds for the critical path, it also holds for the branch we added. 
\end{proof}

In the previous proof, it is possible that $b=\sigma^{n+1}\nu$. In this case the Hubbard tree $\Hub(\nu)$ is finite: $\Hub(\nu)=H_{n+1}=H_n$. In particular, if $\nu$ is $\star$-periodic with period $p$, then $\Hub(\nu)=H_{p-1}$. This implies that Fatou intervals only occur in finite trees. However, we will show later that the Hubbard tree can have infinitely many endpoints as well (but only for non-periodic $\nu$). 

\begin{definition}
\label{Def:Htree}
For each kneading sequence $\nu$, we define its \emph{Hubbard tree}
\[
\Hub(\nu)=\bigcup_{n=0}^\infty \sigma^n(C_\nu)
\]
with the following topology:
\begin{itemize}
\item
if $\nu$ is non-periodic, the topology on $\Hub(\nu)$ is the subspace topology of $X_\nu$;
\item
if $\nu$ is $\star$-periodic with period $p$, the Hubbard tree is given the same topology as $H_{p-1}$.
\end{itemize}
\end{definition}

These two cases are less different than it might seem. In fact, if $\Hub(\nu)$ is a finite tree, that is $\Hub(\nu)=H_n$ for some $n$, then we can define the topology on $\Hub(\nu)$ equivalently as the topology coming from $H_n$. This is the content of the following lemma (which extends the corresponding result on the critical path, as stated in Lemma~\ref{Lem:path_top}).

\begin{lemma}[Equivalent topologies]
\label{Lem:top_equiv}
If the Hubbard tree is finite, i.e.\ $\Hub(\nu)=H_n$ for some $n$, then the topologies of $H_n$ and of $\Hub(\nu)$ are equivalent. 
\end{lemma}
\begin{proof}
This is true by definition in the $\star$-periodic setting. 

If $\nu$ is non-periodic and $\Hub(\nu)$ is finite tree, then $\Hub(\nu)$ does not contain Fatou intervals. Pick an open basis set $U_k(a)\subset X_\nu$; then $U_k(a)\cap \Hub(\nu)$ is an open basis set for $\Hub(\nu)$. 
By Lemma~\ref{Lem:Hn}\eqref{it:dense} and the fact that the number of precritical points up to a given depth is finite on the critical path, and thus on finite trees, we can find an open neighborhood $U$ of $a$ (in the topology of $H_n$) whose boundary points are precritical points of depth at least $k+1$ such that all precritical points in $U$ have depth at least $k+1$ as well. This implies $U \subset U_k(a)$. 

For the converse, pick an open set $U \subset H_n$ and $a\in U$. Let $J \in \mathcal I_j$ for $j\le n$ be any interval containing $a$, necessarily as an interior point. Since $\nu$ is non-periodic, there exists $k(J)$ such that $U_{k(J)}(a)\cap J \subset U\cap J$ because the same holds for the critical path by Lemma~\ref{Lem:path_top}. Since the number of intervals $J$ is finite in $H_n$, we can choose $k=k(J)$ uniform over all $J$. We conclude $U_k(a) \subset U$. 

Combining these two arguments, every neighborhood in the topology of $H_n$ contains a neighborhood in the topology of $\Hub(\nu)$ and vice versa, so the topologies are equivalent. 
\end{proof}

\begin{lemma}[Continuity]
\label{Lem:Hcont}
The map $\sigma: \Hub(\nu) \mapsto \Hub(\nu)$ is continuous.
\end{lemma}
\begin{proof}
In the non-periodic case, this is clear by Lemma~\ref{Lem:X_continuous}. 

In the $\star$-periodic case, the tree is finite and every branch point has finite valency, so it suffices to show continuity along every interval $U$ in the tree.

If $U\subset\Hub(\nu)$ is an interval that does not contain the critical point, then $\sigma\colon U\to\sigma(U)$ is an order-preserving bijection, hence a homeomorphism. Thus for any open $V\subset\Hub(\nu)$ that does not contain $\nu$, the preimage is open.

It remains to consider the case that $V\ni\nu$, and we may shrink $V$ so that it does not contain a branch point of $\Hub(\nu)$, except possibly $\nu$ itself (we will show below that $\nu$ is always an endpoint). Then $\sigma^{-1}(V)$ consists of $\star\nu$ together with finitely many open intervals that terminate at $\star\nu$, and every edge in $\Hub(\nu)$ that ends at $\star\nu$ intersects $\sigma^{-1}(V)$ in such an open interval. Therefore, $\star\nu$ is an interior point of $\sigma^{-1}(V)$.
\end{proof}

\Newpage

\begin{theorem}[The Hubbard Tree]
\label{Thm:Hubbard_tree}
The Hubbard tree $\Hub(\nu)$ is a tree and has the following properties:
\begin{enumerate}
\item 
\label{Item:HubEndpoints}
All endpoints of the tree are postcritical points, i.e.\ one of the points $\sigma^k(\nu)$ for $k\ge 0$. There are two possibilities:
\begin{itemize}
\item
the tree has finitely many endpoints; if $\kappa$ is this number of endpoints, then the endpoints are exactly the points $\nu$, $\sigma\nu$, $\sigma^2\nu$, \dots, $\sigma^{\kappa-1}\nu$, and no others. In this case, the tree is compact.
\item
the tree has infinitely many endpoints, and these are exactly the points $\sigma^k\nu$ for $k\ge 0$. 
\end{itemize}
In particular, $\nu$ is always an endpoint. 

\item 
\label{Item:HubSigma}
$\Hub(\nu)$ is $\sigma$-invariant, and $\sigma$ is injective on every connected subset that does not contain the critical point $\star\nu$.
\item 
\label{Item:Diff}
If $a,b\in\Hub(\nu)\cap X_\nu$ and  $c \in [a,b]\cap X_\nu$, then $\diff(a,c)\ge \diff(a,b)$.
\item 
\label{Item:PrecDenseHub}
For each itinerary $a\in \Hub(\nu)$, and every component $K$ of $\Hub(\nu)\setminus \{a\}$ that does not contain a Fatou interval with $a$ on the boundary, there exist precritical points $a_n \in K$ converging to $a$. In particular, precritical points are dense in $\Hub(\nu)$ if $\nu$ is non-periodic. \end{enumerate}
\end{theorem}

\begin{proof}
In the finite case, by Lemma~\ref{Lem:top_equiv}, the two introduced topologies for finite trees are equivalent, so we can use Lemma~\ref{Lem:Hn} to show most of the listed properties. 

If $\Hub(\nu)$ has finitely many endpoints, then it is the finite union of paths between endpoints, which are compact, so $\Hub(\nu)$ is compact itself.

If $\diff(a,b)=k$ for $a,b \in \Hub(\nu)$, this means precisely that the depth of the first precritical point in $[a,b]$ is $k$ and thus the path $[a,b]$ can be iterated injectively exactly $k-1$ times (this is the motivation behind Definition~\ref{Def:PrecPointsCritPath}). This implies $\diff(a,c)\ge \diff(a,b)$ if $c \in [a,b]$ and thus claim \eqref{Item:Diff}.

Now let us consider the case that $\Hub(\nu)$ has infinitely many endpoints. 
Pick any $a,b \in \Hub(\nu)$. Then there is an $N$ so that $a,b\in H_n$ for all $n\ge N$, and there is a unique path $[a,b]\subset H_n$ that connects these points, so $\Hub(\nu)$ is path connected. 

To see that it is uniquely path connected, suppose there is a path, say $J$,  that connects $a,b$ within $\Hub(\nu)$ but not in any finite $H_n$. In this proof, paths like $[a,b]$ are always understood with respect to finite trees $H_n$; only $J$ is in $\Hub(\nu)$ but not in any $H_n$.

There must be a sequence of points $a_n\in H_n\sm H_{n-1}$ with $a_n\to a\in\Hub(\nu)$ such that $a_{n+1}\not\in[a_n,a]$ for all $n$ (the path from $a$ to $a_n$ within $H_n$ connects ``the long way'', avoiding the point $a_{n+1}$ close to $a$). We can require that the $a_n$ are linearly ordered so that each $a_n\in[a_{n-1},a_{n+1}]$ and thus $a_n\in[a,a_{n+1}]\in H_{n+1}$. Now $a_n\to a$ requires that $\diff(a_n,a)\to\infty$, while $a_n\in[a,a_{n+1}]$ implies that the sequence $\diff(a,a_n)$ cannot increase; see \eqref{Item:Diff}. This is a contradiction, so $\Hub(\nu)$ is indeed a tree.

The Hubbard tree is $\sigma$-invariant by construction. Suppose there are two distinct points $a$ and $b$ in the same connected subset of $\Hub(\nu)\setminus \{\star\}$ and so that $a$ and $b$ are mapped to the same image under $\sigma$. Since there exists some $n$ with $[a,b]\subset H_n$, this  contradicts Lemma~\ref{Lem:Hn}~\eqref{it:inj}. This shows \eqref{Item:HubSigma}.

For \eqref{Item:PrecDenseHub}, observe first that Fatou intervals do not occur for infinite trees, while the claim has been shown for finite trees. 
Density of precritical points translates to the infinite case. 

We postpone the proof that the critical value is always an endpoint of the Hubbard tree to Lemma~\ref{Lem:CriticalValueEndpoint}. By local injectivity, the endpoints are then precisely the orbit points $\nu, \sigma\nu, \dots, $; when one of them is no longer an endpoint, then no further orbit points of $\nu$ are endpoints until the critical point is reached, if ever. This proves \eqref{Item:HubEndpoints} except for the missing lemma.
\end{proof}

\begin{figure}[htbp]
\includegraphics[width=\textwidth]{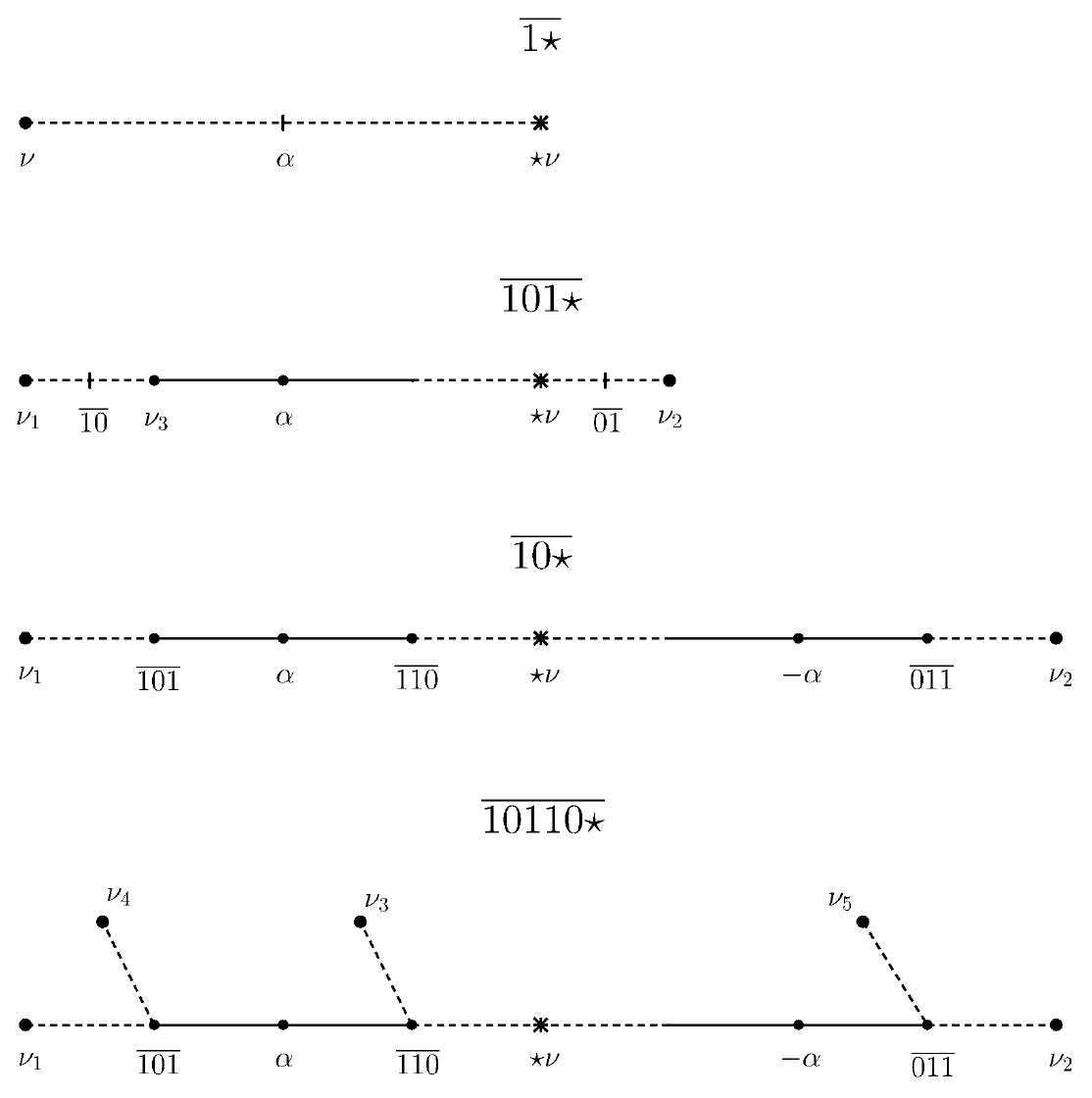}
\caption{The Hubbard trees of the examples considered in Figure~\ref{Fig:crit_paths}. Dotted lines represent major Fatou intervals, vertical bars the orbits of a central itinerary and $-\alpha=\0\overline{\1}$.}
\label{Fig:h_trees}
\end{figure}

\bigskip

\goodbreak

\subsection{The structure of Hubbard trees}

In order to investigate this structure, we need to introduce the concept of global and local arms. 

\begin{definition}[Global and local arms]
Let $\tau$ be a point in $\Hub(\nu)$. The \emph{global arms at $\tau$} are the components of  $\Hub(\nu)\setminus \{\tau\}$; denote them by $G_1(\tau),G_2(\tau),...$ (their number is possibly infinite). These are also called the \emph{branches} of $\Hub(\nu)$ at $\tau$. 

For every global arm $G_i(\tau)$ we define the associated \emph{local arm} $L_i(\tau)$ as an equivalence class of connected open neighborhoods of $\tau$ that do not contain $\star\nu$, restricted to $G_i(\tau)$. We say that the $\sigma$-image of a local arm $L_i(\tau)$ is the local arm $L_j(\tau)$ if $\sigma(U)\subseteq G_j(\tau)$ for all $U\in L_i(\tau)$.  
\end{definition}

Intuitively, a local arm is a ``tangent vector'' at $\tau$ within the associated global arm. We say $L_i(\tau)$ \emph{points towards $\rho$} if $\rho \in G_i(\tau)$. 

Each local arm $L_i(\tau)$ always contains a unique maximal element. If $G_i(\tau)$ does not contain the critical point, this element is simply $G_i(\tau)$. For the unique global arm $G_i(\tau)$ that contains the critical point, the maximal element is the connected component of $G_i(\tau)\setminus \{\star\nu\}$ with boundary point $\tau$. We denote this maximal element by $G_i^1(\tau)$. Note that it is then sufficient to look at the image of the maximal element to determine the image of a local arm.

In the following lemma one should keep in mind that a priori the number of local or global arms at some point $\tau$ could be infinite.

\begin{lemma}[Image of local arms]
\label{Lem:local_arm}
If $\tau$ is not the critical point, then the map $\sigma$ is an injection from the set of local arms of $\tau$ to the set of local arms at $\sigma(\tau)$. 
\end{lemma}
\begin{proof}
Suppose two local arms $L_{i_1}(\tau)$ and $L_{i_2}(\tau)$ both map to $L_j(\sigma(\tau))$ under $\sigma$, then $G_{i_1}^1(\tau)$ and $G_{i_2}^1(\tau)$ both map to  $G_{j}(\sigma(\tau))$ under $\sigma$. But since $\tau$ is not the critical point, $G_{i_1}^1(\tau)\cup \{\tau\}\cup G_{i_2}^1(\tau)$ is injectively mapped to a subtree of $\Hub(\nu)$ where $\sigma(\tau)$ is not an endpoint of the subtree. So this set can not be contained in $G_{j}(\sigma(\tau))$ and we obtain a contradiction.
\end{proof}

\begin{lemma}[Critical value endpoint]
\label{Lem:CriticalValueEndpoint}
The critical value is always an endpoint of the Hubbard tree.
\end{lemma}

\begin{proof}
If the Hubbard tree has finitely many endpoints, this result is well known: if $x$ is not an endpoint, then $\sigma x$ cannot be an endpoint unless $x$ is the critical point. Therefore, if the critical value is not an endpoint, then no postcritical point can be an endpoint. But the Hubbard tree is the finite union of intervals that are bounded by postcritical points, a contradiction. 

The argument is more subtle for a kneading sequence $\nu$ such that its Hubbard tree $\Hub(\nu)$ is infinite. In this case, it is a priori conceivable that the critical value, and eventually all postcritical points, become interior points as further edges of the tree are added in the infinite union $\Hub(\nu)=\bigcup_n\sigma^n[\star\nu,\nu]$. 

Suppose $\nu$ is not an endpoint. Then there must be a $n\ge 1$ such that $\nu\in[\star\nu,\sigma^n\nu]$.

For arbitrary $k\ge 0$, denote by $T_k$ the global arm of $\Hub(\nu)\sm\{\sigma^k\nu\}$ that contains the critical point and let $T_k^1$ be the component of $T_k \setminus \{\star\nu\}$ whose boundary contains $\sigma^k \nu$. 
Let $S_k:=\Hub(\nu)\sm (T_k\cup\{\sigma^k\nu\})$; the set $S_k$ could possibly consist of several components. 
Denote by $s_k$ the minimal depth of precritical points in $S_k$ (not counting the critical value). Since precritical points are dense in infinite trees, $s_k \ge 2$ is always well-defined.

Our claim is that for all $k\ge 0$, we have $s_k < s_0$ or $s_k=s_0$ and $S_0 \subseteq S_k$. Then it follows that $\nu$ cannot lie in the interior of the path $[\star\nu,\sigma^k\nu]$ and we are done. 

We prove the claim by induction on $k$; the case $k=0$ is clear. So suppose the statement holds for some $k$. By Lemma~\ref{Lem:local_arm}, at most one local arm at $\sigma^k \nu$ can map to the local arm associated to $T_{k+1}$. Depending on whether this local arm is or is not associated to $T_k$, we either have  $\sigma(S_k)\subset S_{k+1}$ or  $\sigma(T_k^1) \subset S_{k+1}$. In the former case, we have $s_{k+1} < s_k$, which completes the inductive claim. In the latter case, $S_{k+1}$ contains $(\sigma^{k+1}\nu,\nu]$, but not the critical point, so $\sigma^{k+1}\nu $ must lie on the critical path and hence $S_0 \subsetneq S_{k+1}$ and $s_{k+1}\le s_0$. 
\end{proof}

\Newpage

\begin{theorem}[No Wandering Branch Points]
\label{Thm:NoWanderingBranchPoint}
Every branch point of $\Hub(\nu)$ is always periodic or preperiodic, except if it is a precritical point (only if $d>2$).
\end{theorem}

The following proof is inspired by Thurston's \emph{No Wandering Triangles Theorem} for quadratic minor laminations; see \cite[Theorem~II.5.2]{ThurstonLami}.

\begin{proof}
Let $\tau$ be a branch point of $\Hub(\nu)$. If $\tau$ is the critical point, then the number of components of $\Hub(\nu)\sm\{\tau\}$ can be at most $d$ because the critical value $\nu$ is an endpoint (Lemma~\ref{Lem:CriticalValueEndpoint}). This implies that the number of components at any precritical point can be at most $d$. Thus precritical points can be branch points only if $d>2$.

Now suppose that $\tau$ is a branch point that is not precritical, and set $\tau_k:=\sigma^k\tau$ for $k\ge 0$. By injectivity of the map of local arms, branch points can only be mapped to branch points. Hence, if the tree is finite, the orbit $(\tau_k)$ is finite and $\tau$ has to be periodic or preperiodic. So for the rest of the proof, we can assume that the tree is infinite and precritical points are dense on it.

The set $\Hub(\nu)\sm\{\tau\}$ consists of at least $3$ components; let $A(0)$, $B(0)$, $C(0)$ be three of these. 
For $k\ge 1$ let $A(k)$, $B(k)$, $C(k)$ be the global arms at $\tau_k$ associated to the $\sigma$-images of the local arms corresponding to $A(k-1),B(k-1),C(k-1)$.
Denote by $a(k)$, $b(k)$ and $c(k)$ the minimal depths of precritical points in $A_k$, $B_k$, and $C_k$, respectively (not counting the critical value). Since precritical points are dense, these numbers are well defined. As before, we have $a({k+1})<a(k)$ except when $a(k)=1$, and similarly for $b({k+1})$ and $c({k+1})$. Set $m(k):=\max\{a(k),b(k),c(k)\}$. 

Let us first consider the case $d<\infty$. Then there are only finitely many precritical points of bounded depths, so there can be only finitely many orbit points $\tau_k\in \Hub(\nu)$ for which $m(k)$ satisfies a given bound. Thus if $\tau$ is not eventually periodic, we must have $m(k)\to\infty$. This requires that, while each of the sequences $a(k)$, $b(k)$, $c(k)$ is locally decreasing until it reaches the value $1$, the next value after $1$ must often enough be large enough so that it stays large until one of the other two index sequences becomes $1$ and then large itself. 
In particular, $m(k)$ is decreasing for all indices $k$ for which $a(k)$, $b(k)$, $c(k)$ are all greater than one, but diverging nonetheless.
We show that this is impossible.

There must be infinitely many indices $k$ such that 
\[
m(k)>\max\{m(0),m(1),\dots,m({k-1})\}
\;.
\]
Let $k_j+1$ be the subsequence of such indices. In particular, the subsequence $m({{k_j}+1})$ is strictly increasing. 

Fix a time $k_j$. Then exactly one of $a({k_j})$, $b({k_j})$, $c({k_j})$ must be equal to $1$. Without loss of generality, say $a({k_j})=1$, so $\star\nu\in A(k_j)$ and $m(k_j+1)=a(k_j+1)$. Let $S(k_j)$ be the component of $\Hub(\nu)\sm \sigma^{-1}(\tau_{k_j+1})$ that contains the critical point. Then $S(k_j)\subset A({k_j})$. Note for later use that the minimal depth of precritical points in $S(k_j)\setminus\{\star\}$ is then $a(k_j+1)+1=m(k_j+1)+1$.

Now consider the point $\tau_{k_{j+1}}$ at time $k_{j+1}$. If $\tau_{k_{j+1}}\not\in S(k_j)$ then $S(k_{j+1})\supset S(k_j)$, but this would imply $m(k_{j+1}+1)\le m(k_j+1)$ in contradiction to the construction. Therefore $\tau_{k_{j+1}}\in S(k_j)$. Among the three branches $A(k_{j+1})$, $B(k_{j+1})$ and $C(k_{j+1})$ at $\tau_{k_{j+1}}$, one contains the critical point, and at most one other branch can contain a single element of $\sigma^{-1}(\tau_{k_j+1})$ (the other elements are in the same component as the critical point). Therefore, one of the three branches $A(k_{j+1})$, $B(k_{j+1})$ and $C(k_{j+1})$ is completely contained in $S({k_j})\setminus \{\star\nu\}$. 

So the minimal depth of precritical points in that component is at least $m(k_j+1)+1$ as noted earlier. This is a contradiction because the first time when an index greater than $m(k_j+1)$ appears is by definition at time $k_{j+1}+1$. 

For the case $d=\infty$, interpreting the symbols in $A_\infty$ as integers, we assign to every precritical point $\rho=\mathtt e_1...\mathtt e_{n-1}\star\nu$ of depth $n\ge 1$ a \emph{size}:
\[
|\rho|:=n+\sum_{i=1}^{n-1} |\mathtt e_i|.
\]
Then $|\rho|=1$ if and only if $\rho=\star\nu$. And $|\rho|>1$ implies $|\sigma(\rho)|<|\rho|$. This notion of size of a precritical point is not canonical (unlike the depth of a critical point) because it changes under permuations of $A_\infty$. Nevertheless, it has all the properties of the depth we have used in the proof above: define $a(k)$ as the smallest size of precritical points in $A_k$. Then $a(k)=1$ or $a(k+1)<a(k)$. Since there are only finitely many precritical points of bounded size, once again $\limsup_{k\to \infty} m(k)=\infty$. 
\end{proof}

\begin{lemma}[Number of branches at periodic orbit]
\label{Lem:NumberBranchesConstant}
If a periodic orbit in $\Hub(\nu)$ does not contain the critical point, then all points on the orbit have the same number of branches.
\end{lemma}
\begin{proof}
Since $\sigma\colon\Hub(\nu)\to\Hub(\nu)$ is locally injective except at the critical point, it follows that the number of arms (local or global) at any point $x\in\Hub(\nu)$ can only increase, except at the critical point. So if a periodic orbit does not contain the critical point, then the number of branches at each of its points is a periodic and monotone sequence, hence constant.
\end{proof}


\subsection{Characteristic points and types of branch points}

\begin{definition}[Characteristic point]
A point $\tau\in\Hub(\nu)$ is called a \emph{characteristic point} if a single component of $\Hub(\nu)\sm\{\tau\}$ contains the critical point and all points on the orbit of $\tau$ (except $\tau$ itself when $\tau$ is periodic), but not the critical value. 
\end{definition}

Consequently, a characteristic point $\tau$ lies on the critical path of $\nu$ and the path from $\nu$ to any point in the orbit of $\tau$ contains $\tau$. The critical value $\nu$ is always characteristic.

\begin{lemma}[Return to critical path]
\label{Lem:return}
If $\tau \in \Hub(\nu)$ is not an endpoint or precritical point of the Hubbard tree, then for every local arm $L$ at $\tau$, there exists an iterate of $\sigma$ such that $\tau$ is mapped to the critical path and the image of $L$ points towards the critical value.
\end{lemma}
\begin{proof}
Let $L$ be any local arm at $\tau$ and, for $k\ge 0$, let $A_k$ be the global arm at $\sigma^k\tau$ that contains $\sigma^k(L)$.
Since $A_0$  contains a precritical point, there exists $t_1\ge 0$ such that $A_{t_1}$ contains the critical value. We claim that there exists $t_2\ge t_1$ such that $A_{t_2}$ contains the critical value but not the critical point. Then $\sigma^{t_2}(\tau)$ must lie on the critical path, which finishes the proof. 

If the claim is false, then for all $t\ge t_1$ the global arm $A_t$ contains the critical value and the critical point. Since $\tau$ is not an endpoint or a precritical point, $\sigma^{t_1}(\tau)$ is not an endpoint and there exists a global arm $B_{t_1}$ different from $A_{t_1}$. Now $B_{t_1}$ contains precritical points and is mapped injectively until it contains the critical point, which must happen after a finite number of iterations $s$. Thus, the two corresponding local arms are mapped to the same local arm under $\sigma^s$, which contradicts Lemma~\ref{Lem:local_arm}. 
\end{proof}


\begin{lemma}[Characteristic point]
\label{Lem:ExistenceCharacteristic}
Periodic points have the following properties.
\begin{enumerate}
\item
Every periodic orbit in $\Hub(\nu)$ has a characteristic point, except if it consists of endpoints.
\item
\label{it:glo_char}
Every global arm at the characteristic point maps homeomorphically onto its image under the first return map, except possibly if the associated local arm maps to the  local arm towards the critical point or towards the critical value.
\item
\label{it:loc_ret}
Specifically, the local arm at the characteristic point to the critical point maps, under the first return map of the characteristic point, to the local arm either to the critical point or to the critical value.
\item
\label{it:loc_char}
Moreover, every local arm at every point on the periodic orbit eventually maps to the local arm  at the characteristic point towards the critical value, or to the one towards the critical point. 
\item
\label{Item:NumberArmsFinite}
Every periodic orbit has finitely many local arms.
\end{enumerate}
\end{lemma}

If a periodic orbit contains an endpoint, then by Lemma~\ref{Lem:NumberBranchesConstant} it must either contain the critical point (so the critical point is periodic), or it must be an orbit of endpoints (remember that on a Hubbard tree, every endpoint is postcritical, 
so in this case the critical orbit is strictly preperiodic and terminates in a cycle of endpoints).

\begin{proof}
Consider a periodic point $\tau \in\Hub(\nu)$ that is not an endpoint. If $\tau$ is precritical, then $\nu$ must be $\star$-periodic and $\tau$ is on its orbit. In this case, there is nothing to show and we can exclude this case in the sequel. 

By Lemma~\ref{Lem:return}, there exists an iterate $k$ such that $\sigma^k(\tau)$ is on the critical path. There might be more than one iterate of $\tau$ on the critical path, so denote by $\tau_1$ the one closest to the critical value. If there is a characteristic point on the orbit of $\tau$, it must be $\tau_1$; we need to exclude that orbit points of $\tau$ exist on a branch of $\tau_1$ that is disjoint from the critical path. 

Let $H'\subset\Hub(\nu)$ be the subtree spanned by the orbit of $\tau_1$, and let $\tau_0$ be the periodic preimage of $\tau_1$, that is the point in the orbit of $\tau_1$ such that $\sigma(\tau_0)=\tau_1$. Our next claim is that $\sigma(H')\subset H'\cup[\tau_1,\nu]$. Indeed, every point in $H'$ lies on the path $[\tau_0, \tau_i]$ for some point $\tau_i$ on the orbit of $\tau_1$. The claim follows from the following observation:
\[
\sigma([\tau_0, \tau_i])\subseteq  [\tau_1,\tau_{i+1}]\cup [\tau_1,\nu] \cup [\nu, \tau_{i+1}] \subseteq [\tau_1,\tau_{i+1}]\cup [\tau_1,\nu]
\subset H'\cup[\tau_1,\nu]
\;.
\]
Now we show that $\tau_1$ is an endpoint of $H'$. Indeed, if some point $\tau_i$ has $k$ branches in $H'$, then $\sigma(\tau_i)$ has at least $k$ branches in $\sigma(H')$. Since $\sigma(H')\subset H'\cup[\tau_1,\nu]$, this means that $\sigma(\tau_i)$ must have at least $k$ branches in $H'$, except if $\sigma(\tau_i)=\tau_1$; in the latter case, the branch to $\nu$ may be missing in $H'$ and there are at least $k-1$ branches. Since the tree $H'$ must have some endpoints, this implies that $\tau_1$ is indeed an endpoint and all other orbit points have one or two branches. It follows that $\tau_1$ is characteristic.

Let $G$ be a global arm at $\tau_1$ with local arm $L$. If $G$ is not mapped homeomorphically onto its image under the first return map, then there exists $k$ smaller than the period of $\tau_1$ such that $\sigma^k(G)$ is mapped homeomorphically to the global arm at $\sigma^k(\tau_1)$ that contains the critical point. Thus, $\sigma^{k+1}(L)$ points to $\nu$; if it is not based at $\tau_1$, then it also points to $\tau_1$. Iterating this argument, we get that the image of $L$ under the first return map points to the critical point, to the critical value, or to another orbit point of $\tau$. But since $\tau_1$ is characteristic, the last case also implies that the local arm points to the critical point. This shows \eqref{it:glo_char}. Since the global arm at $\tau_1$ that contains the critical point is clearly not mapped homeomorphically, we immediately get \eqref{it:loc_ret}.

Now let $L_0$ be a local arm at $\tau$, then an iterate of $L_0$ is a local arm $L_1$ at $\tau_1$ with corresponding global arm $G_1$. If $G_1$ consists of a single Fatou interval, it must eventually be mapped to a Fatou interval containing the critical point. Otherwise, $G_1$ must contain a precritical point in its interior, so the first return map cannot map $G_1$ homeomorphically indefinitely. Thus, by \eqref{it:glo_char}, the local arm $L_0$ is eventually mapped to the local arm at $\tau_1$ pointing towards the critical point or critical value, yielding \eqref{it:loc_char}. 

Finally, we prove \eqref{Item:NumberArmsFinite}. It follows from Property~\eqref{it:loc_ret} that, at the characteristic point, the two local arms pointing to the critical point and to the critical value are periodic, possibly on different orbits (of local arms) and of different periods. By construction of the Hubbard tree, every local arm in the orbit of $\tau$ must be on the orbit of these two local arms of the critical path, so the number of local arms is finite. \looseness-1
\end{proof}

In fact, the previous proof shows more about possible types of periodic orbits that do not contain the critical point.

\begin{corollary}[Types of periodic orbits]
\label{Cor:orbits}
Every periodic orbit that does not contain the critical point has exactly one of the following types:
\begin{description}
\item [endpoints]
it is an orbit of endpoints;
\item [primitive type]
each point has exactly two branches that are fixed (locally) by the first return map;
\item [satellite type]
each point has two or more branches that are permuted transitively and cyclically by the first return map (of course the same number along the orbit);
\item [evil type]
each point has three or more branches that are permuted non-transitively, such that the first return map at the characteristic point fixes the arm towards the critical point,  and permutes all other arms transitively and cyclically.
\end{description}
\end{corollary}

Milnor calls orbits of endpoints \emph{trivial}, except in the special case of period $1$ (the $\beta$ fixed point). The notation of \emph{evil} was introduced in \cite{HenkDierkTreesAdmissible}.

\begin{proof}
By Lemma~\ref{Lem:NumberBranchesConstant}, the number of branches at a periodic orbit is constant along the orbit. 

If the number of local arms at each point is $1$, we have an orbit of endpoints.

If the number is $2$, then the first return map must either fix both local arms, or it must exchange them, and we are in the primitive or satellite case.

It remains to consider the case that the number of local arms is $s\ge 3$. Since local arms are mapped injectively from any periodic point to its image, all local arms must be on periodic cycles of local arms, possibly on different cycles of different periods. 
 
Recall from Lemma~\ref{Lem:ExistenceCharacteristic} that every local arm at the characteristic point must map either to the local arm towards the critical point or towards the critical value. Here we distinguish two cases: the local arms at the characteristic point towards the critical point and towards the critical value are on the same orbit, or they are on different orbits. 

Suppose first that they are on the same orbit. Then all local arms at the characteristic point are on the same orbit, and we have the satellite type.

The final case is that the local arms at the characteristic point towards the critical point and towards the critical value are on different orbits. 
By this assumption and Lemma~\ref{Lem:ExistenceCharacteristic} \eqref{it:loc_ret}, the first return map fixes the local arm towards the critical point. Thus all other local arms at the characteristic point must lie in the orbit of the local arm towards the critical value and we have the evil type. 
\end{proof}

By Lemma~\ref{Lem:ExistenceCharacteristic}~\eqref{Item:NumberArmsFinite} every periodic point has finite valency. In fact, branch points have finite valency in many further cases, but not all.

\begin{corollary}[Finite valency]
Every branch point of $\Hub(\nu)$ disconnects the tree into finitely many components, except possibly precritical points when $d=\infty$.
\end{corollary}

\begin{proof}
By Theorem~\ref{Thm:NoWanderingBranchPoint} every branch point is either precritical or eventually periodic. Periodic points have finitely many arms, and the same follows for preperiodic points. 

The critical value is always an endpoint (Lemma~\ref{Lem:CriticalValueEndpoint}), so every precritical point has at most $d$ components.
\end{proof}

We show in Example~\ref{Ex:InfiniteValency} that Hubbard trees can have branch points with infinitely many arms in the only possible case that is allowed by this lemma: if $d=\infty$ then there may be precritical points with infinitely many arms. Such Hubbard trees differ from the standard definition in three aspects: they are postcritically infinite, they have infinite degree, and the trees have infinitely many endpoints. 

The valency of branch points has a useful combinatorial implication.

\begin{corollary}
\label{Cor:per_diff}
Let $\tau$ be a characteristic, periodic branch point with period $p$ and $k$ branches. Then $\diff(\nu, \tau) \ge (k-2)p +1$. In particular, the global arm at the characteristic point that contains the critical value maps forward homeomorphically for at least $(k-2)p$ iterations. 
\end{corollary}
\begin{proof}
Whether or not $\tau$ is of evil or of satellite type, the first time the local arm at $\tau$ pointing towards $\nu$ is mapped under the first return map to the local arm towards the critical point or the critical value is at $k-1$ iterations. Hence by Lemma~\ref{Lem:ExistenceCharacteristic}~\eqref{it:glo_char}, the global arm at $\tau$ containing $\nu$ is mapped homeomorphically under $\sigma^{(k-2)p}$. 
\end{proof}

\Newpage

\subsection{Consistency with the classical definition}

We end this section by summarizing some properties of our Hubbard trees to show that our construction truly extends the classically known Hubbard trees in the cases when the latter exist.

\goodbreak

\begin{theorem}
    The Hubbard tree $T=\Hub(\nu)$ has the following properties:
    \begin{enumerate}
        \item \label{it:cont_surj} $\sigma:T\mapsto T$ is continuous and surjective.
        \item \label{it:inverse_im} Every point in $T$ has at most $d$ inverse images under $\sigma$. 
        \item \label{it:local_homeo} At every point other than the critical point, the map $\sigma$ is a local homeomorphism onto its image 
        \item \label{it:endpoints} All endpoints of $T$ are on the critical orbit. 
        
        \item \label{it:expansivity} Expansivity: If any $x$ and $y$ with $x\neq y$ are branch points or points on the critical orbit, then there is an $n \ge 0$ such that $\sigma^n([x,y])$ contains the critical point. 
    \end{enumerate}
\label{Thm:tree_properties}
\end{theorem}
\begin{proof}
\eqref{it:cont_surj} Continuity is shown in Lemma~\ref{Lem:Hcont}. To show surjectivity, note from the definition of the Hubbard tree that is suffices to show that $[\star\nu,\nu]\in \sigma(T)$. We distinguish whether there exists an image $\nu_k$ of $\nu$ (which may be $\nu$ itself in the periodic setting) that starts with $\1$. If this is the case, let $\nu_l$ be an image of $\nu$ that does not start with $\1$. By Lemma~\ref{Lem:alpha}, the fixed point $\alpha \in [\star\nu,\nu]$ and we may assume that $\nu_k$ lies on or branches off inside $[\alpha, \nu]$, otherwise take $\sigma(\nu_k)$ instead. Now $[\star\nu, \alpha] \subset [\nu_l, \nu_k] \subset \sigma(T)$ and hence also $[\alpha, \nu] \subset \sigma(T)$. In the other case, just note that $[\star\nu, \nu] \subset [\nu_3, \nu] = \sigma([\star\nu, \nu_2]) \subset \sigma(T)$. 

Together with the fact that $\sigma$ is injective on all connected subsets not containing $\star\nu$ (see Theorem~\ref{Thm:Hubbard_tree}), we also get \eqref{it:local_homeo}.  The theorem shows \eqref{it:endpoints} too. \eqref{it:inverse_im} is clear as this even holds in $X_\nu$. By construction, all intervals (including Fatou intervals) must contain a precritical point, this shows \eqref{it:expansivity}.
\end{proof}

\begin{corollary}
 Let $\nu$ be a periodic or preperiodic kneading sequence. The Hubbard tree $\Hub(\nu)$ as defined in Definition~\ref{Def:Htree} coincides as dynamcial system with the established notions of the Hubbard tree. 
\end{corollary}
\begin{proof}
    Theorem~\ref{Thm:tree_properties} states that the Hubbard tree in our construction is considered a Hubbard tree according to Definition 2.2. in \cite{HenkDierkTreesAdmissible}. By Corollary 4.20 in \cite{HenkDierkTreesAdmissible}, the tree with those mentioned properties is unique. 
\end{proof}

Of course every tree, finite or not, can be embedded into the plane. It is, however, not always possible to embed the tree so that the dynamics respects the embedding. The classical counterexample is the tree for kneading sequence $\ovl{\1\0\1\1\0\star}$ of period $6$: it has a periodic orbit of period $3$ for which the first return map fixes one local arm and permutes the other, and this is impossible to embed into $\C$ so that the dynamics respects the cyclic order (or, equivalently, the enbedding into the plane). 

Branch points are (pre)-periodic or pre-critical, but preperiodic and precritical branch points are never an obstruction to embeddability because their embedding can always be made consistent by taking (iterated) preimages from periodic branch points or from the critical point, respectively.

\begin{proposition}[Embedding of tree into $\C$]
\label{Prop:embedding}
The Hubbard tree $\Hub(\nu)$ can be embedded into $\C$ in such a way that the dynamics respects the cyclic order at periodic branch points if and only if all periodic branch points are of satellite type.
\end{proposition}

In other words, in view of Lemma~\ref{Cor:orbits} the only obstruction for a tree to be dynamically embeddable into $\C$ is the existence of evil periodic orbits.

\begin{proof}
The existence of a single branch point of evil type certainly makes it impossible to embed the tree consistently into $\C$. If all periodic branch points are of satellite type, then all local arms are permuted transitively by the first return map, and one possible embedding is so that the first return maps sends every local arm to the next one in counterclockwise order (so that the combinatorial rotation number is $1/q$ where $q$ is the number of arms at the periodic orbit). This embedding is clearly consistent.
\end{proof}

In \cite[Section~3]{HenkDierkTreesAdmissible}, \cite[Section~5]{SymDyn} a formula is given for the number of dynamically consistent embeddings: if $q_i$ are the number of arms at periodic branch orbits (indexed by $i$) and $\phi$ denotes the Euler $\phi$ function (so that $\phi(n)$ counts the numbers in $\{1,\dots,n\}$ coprime to $n$), then the number of embeddings is $\prod_i \phi(q_i)$. In \cite{HenkDierkTreesAdmissible} the number of branch orbits is always finite, but the formula applies in our case as well, yielding an infinite value when there are infinitely many branch orbits.

\begin{remark}
For exponential maps, the combinatorics has been investigated in \cite{Escaping,LasseDierkCombinatorics,LasseDierkInventiones}. In particular, there is a growth condition on sequences that describe external angles: a sequence in $(a_n)\in\Z^\N$ has to be \emph{exponentially bounded} in order to be realized as external address (external angle) of an exponential map: that means there exists an $x\in\R$ such that $|a_n|$ is bounded by the $n$-th iterated exponential of $x$. The growth condition for an external angle is the same as for the associated kneading sequence, so in order for a kneading sequence with $d=\infty$ to be realized by a complex exponential map it also has to be exponentially bounded. We keep this discussion brief here because no such restriction exists for our trees: they exist for every kneading sequence in $\Z^\N$ without any growth condition. 

In fact, as abstract trees without embedding any permutation of the symbol set $\Z$ will lead to an isomorphic tree. Therefore, the first $n$ entries in $\nu$ will use no more than $n$ symbols, so every tree can also be realized by a linearly bounded kneading sequence. 
\end{remark}

\begin{figure}[htbp]
\hspace*{0cm}
\includegraphics[width=1.0\textwidth]{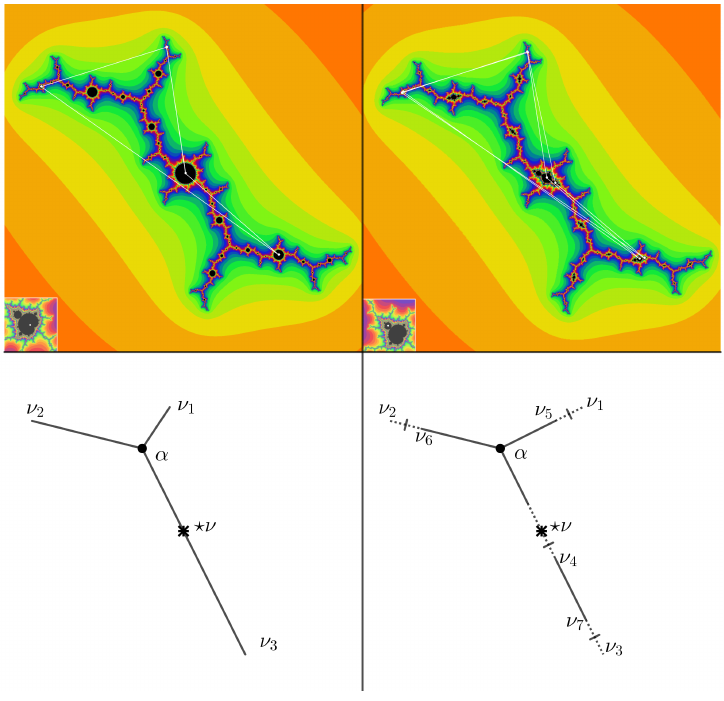}
\caption{\textit{Top left:} A Julia set with kneading sequence $\nu=\ovl{\1\1\0\star}$ with the critical orbit highlighted. \textit{Top right:} The Julia set of its bifurcation $\tilde\nu=\ovl{\1\1\0\0\;\1\1\0\star}$. \textit{Bottom left:} The (abstract) Hubbard tree of $\nu$. \textit{Bottom right:} The Hubbard tree of $\tilde\nu$. Several prominent Fatou intervals are represented by dotted lines, bars between them show the orbit of the central itinerary $\ovl{\1\1\0\0}$.  }
\label{Fig:julia_tree}
\end{figure}

\Newpage

\section{Properties and examples}
\label{Sec:examples}

In this section, we discuss various properties that Hubbard trees may have, and that they are independent of each other: in particular recurrence of $\nu$ and the existence of infinitely many endpoints. A kneading sequence is \emph{recurrent} if the critical value is a limit point of its forward orbit. 
For simplicity, we give examples only in the simplest non-trivial case of degree $2$. It is well known that every kneading sequence $\nu$ in degree $2$ can be equivalently encoded in terms of an \emph{internal address} $I(\nu)$: this is a strictly increasing, finite or infinite sequence of integers $S_0\to S_1\to ...$ that is recursively defined as follows: 

\Newpage

\begin{itemize}
\item $S_0=1$, $\varsigma_0:=\ovl \1$. 
\item If $\diff(\varsigma,\varsigma_n)=\infty$, then $I(\varsigma)$ is finite and its last entry is $S_n$.
\item Otherwise, set $S_{n+1}=\diff(\varsigma,\varsigma_n)$. Let $\varsigma_{n+1}$ be the unique $S_{n+1}$-periodic sequence in $\1A_2^\infty$ that first differs from $\varsigma_{n}$ at position $S_{n+1}$.
\end{itemize}

So a kneading sequence has finite internal address if and only if it is $\star$-periodic. We can also define internal addresses of infinite sequences over $A_d$. For example, both the kneading sequence $\ovl{\1\0\star}$ and the sequence $\ovl{\1\0\0}$ have a finite internal address $1\to 2 \to 3$, but the sequence $\ovl{\1\0\1}$ has an infinite internal address starting with $1\to 2 \to 4 \to 5 \to 7 \to 8 \to...$. 

Before giving examples of Hubbard trees, we present how the internal address is useful to classify bifurcations.

\subsection{Classification of bifurcations} 

\begin{lemma}[Classification of Fatou intervals]
\label{Lem:Fatou}
For each $\star$-periodic sequence $\nu$ there exists a unique itinerary $\omega \in A_d^\infty \cap X_\nu$ such that $[\nu, \omega]$ is a Fatou interval on the critical path of $\nu$. All other Fatou intervals of $\Hub(\nu)$ then have the form $[w\star\nu, w\mathtt e\omega]$, where $w\star\nu$ is a precritical point and $\mathtt e\in A_d$.
\end{lemma}

\begin{proof}
According to the construction of Fatou intervals and since the critical value is an endpoint of $\Hub(\nu)$ (Lemma~\ref{Lem:CriticalValueEndpoint}), there is a unique Fatou interval with $\nu$ on the boundary, say $[\omega,\nu]$ where $\omega$ is a periodic sequence in $X_\nu$. 

Any other Fatou interval $[w'', w\star\nu]$ is mapped homeomorphically  under $\sigma^{|w|+1}$ to a Fatou interval containing $\nu$, and the only candidate is $[\omega,\nu]$. The claim follows.
\end{proof}

We call this sequence $\omega$ the \emph{lower sequence} of $\nu$. It differs from $\nu$ in the fact that all the symbols $\star$ are replaced consistently by the same symbol in $A_d$, and this symbol is uniquely determined by the fact that $\omega$ is a lower sequence.

\begin{definition}[Standard bifurcations]
A bifurcation $\nu$ with base sequence $\mu$ is called a \emph{standard bifurcation} if the period of $\mu$ occurs in the internal address of $\nu$. Otherwise, we call $\nu$ a \emph{non-standard bifurcation}.
\end{definition}

The only bifurcations that occur in complex dynamics are standard bifurcations: non-standard bifurcations always lead to evil periodic orbits that violate the embeddability of the tree into the plane in such a way that the the dynamics respects this embedding; see below.

In Figure~\ref{Fig:julia_tree} we give an example of a Hubbard tree for  $\nu=\ovl{\1\1\0\0\;\1\1\0\star}$ with internal address $1\to 3 \to 4 \to 8$, which is a standard bifurcation. We also show how the tree corresponds to the tree of the base sequence, and the associated Julia sets. In Figure~\ref{Fig:h_trees} we show an example of a non-standard bifurcation $\ovl{\1\0\1\;\1\0\star}$ with internal address $1\to 2 \to 4 \to 5 \to 6$, together with two more standard bifurcations.

\begin{proposition}[Classification of $\star$-periodic kneading sequences]
\label{Prop:bifurcations} 
Let $\nu$ be a $\star$-periodic kneading sequence with period $p$. Then the lower sequence $\omega$ is characteristic and
\begin{enumerate}
\item if $\nu$ is not a bifurcation, then $\omega$ has period $p$ and and all of its preimages have two branches; 
\item if $\nu$ is a standard bifurcation, then $\omega$ has the same period $q$ as the base sequence of $\nu$. The itinerary $\omega$ is a central itinerary and as such has $p/q$ branches and so that the orbit of $\omega$ is of satellite type;
\item if $\nu$ is a non-standard bifurcation, then $\omega$ also has the same period $q$ as the base sequence of $\nu$ but $\omega$ has $p/q+1$ branches and its orbit is of evil type.
\end{enumerate}
Central itineraries (the case where $k=\infty$ in Definition~\ref{Def:PrecPointsCritPath}) occur only in the case of standard bifurcations.
\end{proposition}

This proposition has several interesting implications. Firstly, for $\star$-periodic sequences that are not bifurcations, no preimage of $\omega$ is a branch point. So the introduction of Fatou intervals is indeed ``optional'' in the sense that collapsing them does not alter the tree structure, in contrast to bifurcations. Secondly, kneading sequences that arise from non-standard bifurcations cannot be embedded in a way that respects the dynamics (see Proposition~\ref{Prop:embedding}). Thus, these sequences are not complex-admissible. 

\begin{proof}
Let us first consider the case that $\nu$ is not a bifurcation. Since $\diff(\nu,\omega)=\infty$ and no image of $\sigma^k \nu$ other than $\nu$ itself satisfies $\diff(\nu, \sigma^k \nu)=\infty$, the period of $\omega$ must be a multiple of $p$. And since $\nu$ is an endpoint of the Hubbard tree, it follows that the period has to be exactly $p$. So the local arm at $\omega$ towards $\nu$ is mapped to itself under the first return return map. By Corollary~\ref{Cor:orbits}, the orbit of $\omega$ is of primitive type. So it and all of its preimages have two branches and $\omega$ is clearly characteristic. 

Let us now consider the case that $\nu$ is a standard bifurcation. By construction of the internal address, there exists a precritical point $\rho$ of depth $q$ that coincides with $\nu$ for at least $q$ entries. Since $\nu$ is a bifurcation, $\diff(\nu,\rho)=\infty$. In other words, the case $k=\infty$ in Definition~\ref{Def:PrecPointsCritPath} occurs. The central itinerary then has to be the lower sequence $\omega$ and it is clearly characteristic. Since $[\omega, \rho]$ must be mapped to $[\nu, \omega]$ under $\sigma^q$, the period of $\omega$ is $q$. As central itinerary, all branches at $\omega$ contain Fatou intervals with boundary point $\omega$, so by iterating $[\nu, \omega]$ any multiple of $q$ times, we obtain all $p/q$ branches at $\omega$. They are cyclically permuted, so the orbit of $\omega$ is of satellite type. 

Now let $\nu$ be a non-standard bifurcation. If the internal address of $\nu$ contains any multiple $kq$ for $k\ge 2$ than there would have to be a precritical point  $\rho$ of depth $kq$ coinciding with $\nu$ for $kq$ entries. But just as in the previous case, this would imply that $\nu$ is a standard bifurcation with base sequence having period $kq$ instead. Hence, between the critical value and any precritical point on the critical path with depth a multiple of $q$, there lies a precritical point of lower depth. We can conclude that all precritical points $\rho$ (different from $\nu$) satisfy $\diff(\nu, \rho)<\infty$. Thus, $\nu$ is not adjacent to a central itinerary and by Lemma~\ref{Lem:Fatou}, no central itineraries occur anywhere in the Hubbard tree. 

Thus, the lower sequence $\omega$ occurs as an element of $M_\nu$.  Since $\diff(\nu, \sigma^q \nu)=\infty$ and $\nu$ is an endpoint of the Hubbard tree, we see that $\omega$ must be fixed by $\sigma^q$, so its period is $q$. By iterating $[\nu, \omega]$, we further see that $\omega$ has $p/q \ge 2$ branches of the form $[\sigma^{kp}\nu, \omega]$ that are permuted cyclically by $\sigma^q$, as well as one additional branch towards the critical point. By Corollary~\ref{Cor:orbits}, those are all branches at $\omega$ and the orbit of $\omega$ is of evil type.
\end{proof}

\subsection{A trip to the zoo (i.e.\ to the woods): examples of trees}

Let us start with a trivial observation.

\begin{example}[Non-recurrent, finite tree]
Every strictly preperiodic kneading sequence produces a finite tree, and clearly it is non-recurrent.
\end{example}

There are many sequences that are non-recurrent and produce finite trees, but unlike preperiodic ones they do not need to have a finite critical orbits:

\begin{example}[Non-recurrent, finite tree, postcritically infinite] 
Let $\{w_i\}_{i\in \N}$ be an infinite sequence of words $w_i\in \{\1\1, \1\0\1\}$ such that the kneading sequence $\nu=\1\0\0\1w_1w_2w_3...$ is not (pre)periodic. Then $\nu$ is non-recurrent and has a finite tree, yet containing infinitely many distinct postcritical points. 
\end{example}

Since $\nu$ is the only postcritical point starting with $\1\0\0$, the sequence is non-recurrent. It can be recursively constructed as follows: 

starting with $\nu_0=\ovl{\1\0\0\star}$, given a $\star$-periodic kneading sequence $\nu_k$ with period $p_k$, we follow the internal address of the lower sequence of $\nu_k$ until the entry $p_{k+1}:=p_k+|w_{k}|$. Then $\nu_{k+1}$ is the associated $\star$-periodic kneading sequence with period $p_{k+1}$. The associated parameter of $\nu$ in the Mandelbrot set is then the limit of increasing, real parameters of $\nu_k$ and thus real as well. Hence, its Hubbard tree is an interval, which is clearly finite.

\begin{lemma}[Recurrent kneading sequence via internal address]
If an internal address $1\to S_1\to S_2\dots$ has the property that $\limsup (S_{k+1}-S_k)\to\infty$ then the associated kneading sequence is recurrent.
\label{Lem:recurrent}
\end{lemma}
This is because the sequences $\nu$ and $\sigma^{S_k}(\nu)$ coincide for $S_{k+1}-S_k-1$ entries.

\begin{example}[Recurrent kneading sequence with finite tree]
A famous example is the Morse--Thue-parameter with internal address $1\to 2\to 4\to 8\to\16\dots$. The associated parameter in the Mandelbrot set is the Feigenbaum point, which is a real parameter, so the Hubbard tree is an interval.
\end{example}

\begin{example}[Infinite tree via internal address]
\label{Ex:InfiniteTreeIntAddr}
If an internal address $1\to S_1\to S_2\dots$ has the property that $S_{k+1}/S_k>2$ for infinitely many $k$, then the associated kneading sequence has the property that its Hubbard tree has infinitely many periodic orbits that are branch points, and thus infinitely many endpoints. 

All such kneading sequences are recurrent by Lemma~\ref{Lem:recurrent}. 
\end{example}

While this result does not require that the internal address or its kneading sequence are complex admissible, the reasoning is documented better in this case, and we can find such kneading sequences in the Mandelbrot set. If $a_k:=\lceil S_{k+1}/S_k\rceil$, then the internal address $1\to \dots\to S_k\to S_{k+1}$, and all extended internal addresses are in a $p_k/q_k$-limb of a period $S_k$ component, where $q_k=a_k$ or $q_k=a_k+1$ \cite[Lemma~4.1]{IntAddr}. This implies that upon entering that limb, a periodic cycle with $q_k$ branches is created. This is a branch point if $q_k\ge 3$, and this is guaranteed if $S_{k+1}/S_k>2$. If this happens for infinitely many $k$, then the Hubbard tree has infinitely many periodic branch points on infinitely many branch cycles. And this, in turn, implies that the Hubbard tree must have infinitely many endpoints.

It is also possible to have infinite trees for non-recurrent kneading sequences. 

\begin{example}[Infinite tree, non-recurrent]
\label{Example:InfiniteTreeNonRecurrent}
There are infinite trees for which the kneading sequence is non-recurrent.
\end{example}

Consider the sequence $\nu=\1\,\1\0\,\1\0\0\,\1\0\0\0\,\1\0\0\0\0\,\1...$ where the number of consecutive $\0$'s between $\1$'s keeps increasing by one. This ensures that $\nu$ is postcritically infinite. Since the critical value is the only postcritical point that starts with $\1\1$, the kneading sequence is non-recurrent. In fact, by comparing the itineraries of $\nu$ and $\star\nu$, we see that no other postcritical point can lie on the critical path. By Lemma~\ref{Lem:return}, this implies that all postcritical points are endpoints, so the tree is indeed infinite. See Figure~\ref{Fig:InfiniteNonRecurrent} for a sketch of $\Hub(\nu)$. 

\begin{figure}[htbp]
{
\includegraphics[width=\textwidth]{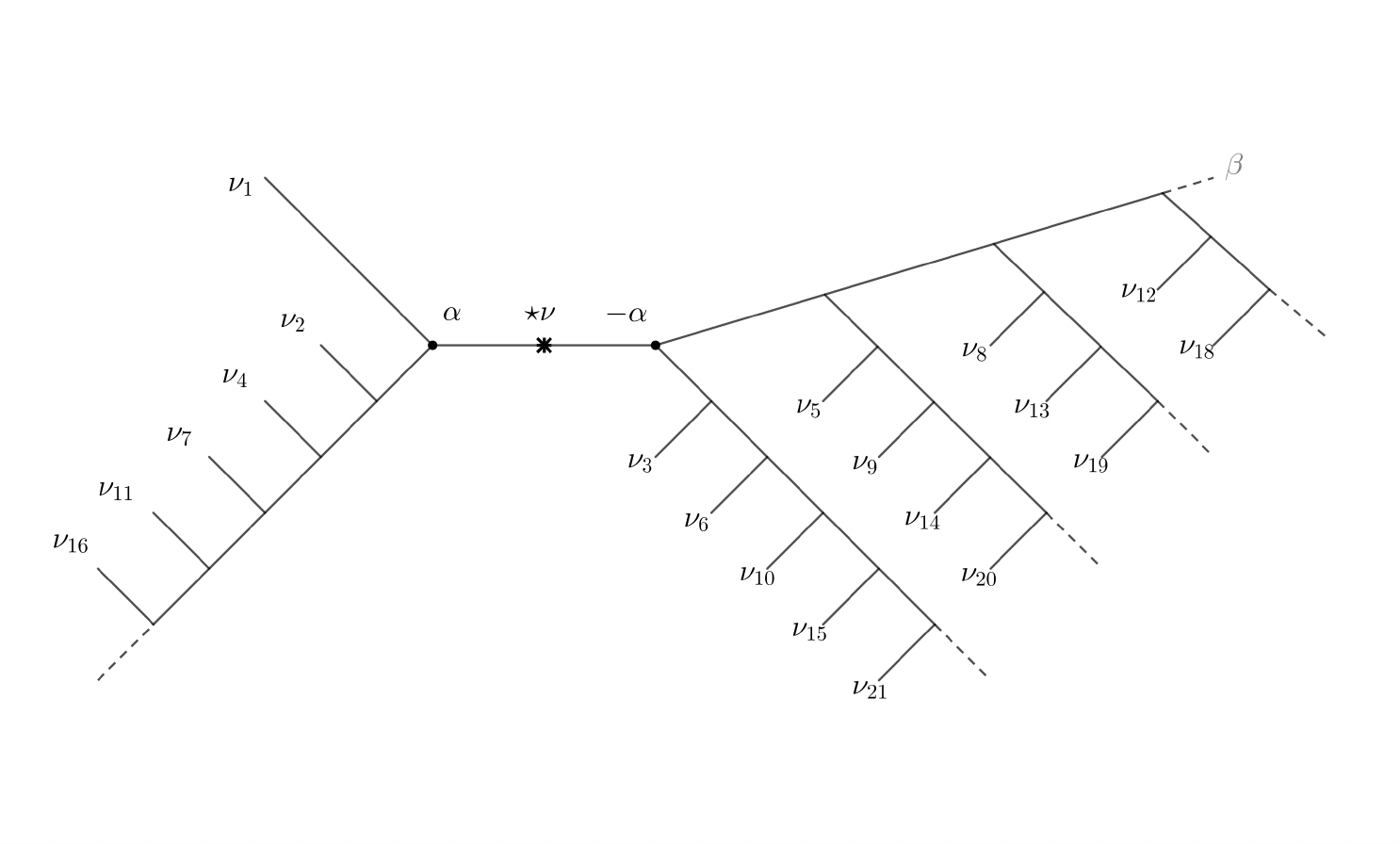}
}\caption{Sketch of an infinite tree with non-recurrent kneading sequence $\nu=\1\,\1\0\,\1\0\0\,\1\0\0\0\,\1\0\0\0\0\,\1...$ . The fixed point $\beta=\ovl\0$ is a limit point of the critical orbit, but does not belong to the tree.}
\label{Fig:InfiniteNonRecurrent}
\end{figure}

Observe that in Example~\ref{Example:InfiniteTreeNonRecurrent} there are infinitely many branch points, but all are on the backwards orbit of the fixed point $\alpha$ --- unlike Example~\ref{Ex:InfiniteTreeIntAddr} where we have infinitely many periodic branch orbits. The following result explains why.

\begin{lemma}[Infinite tree for non-recurrent kneading sequence]
If $\nu$ is non-recurrent then there are only finitely many periodic branch points. 
\end{lemma}
\begin{proof}
Suppose there exist infinitely many periodic branch points. We want to show that $\nu$ is recurrent. Let $\tau_1,\tau_2,...$ be the characteristic branch points on the respective orbits, choose a postcritical point $\nu(\tau_i)$ on a branch of $\tau_i$ disjoint from the critical path and let $\rho_i$ be the precritical point on $[\tau_i, \nu(\tau_i)]$ of lowest depth. 

In the case $d<\infty$, the periods of $\tau_i$ must be unbounded because there are only finitely many points of any given period. Corollary~\ref{Cor:per_diff} implies 
\[
\limsup_{i\to \infty} \diff(\nu, \tau_i)=\infty
\;.
\]
Since also the depth of the $\rho_i$ must be unbounded in this case, we obtain \[
\limsup_{i\to \infty} \diff(\nu, \nu(\tau_i))=\infty
\;,
\] 
i.e.,  $\nu$ is recurrent. 

The case of infinite degree needs a bit more attention. To begin with, suppose that the periods of $\tau_i$ are bounded. Then there exists a subsequence of the $\tau_i$ where all elements have period $p$ for some $p$. But then between two adjacent elements of this subsequence there has to be a precritical point of depth at most $p$. They all lie on the critical path, in contradiction to the fact that the number of those precritical points is finite. Therefore, the periods of the $\tau_i$ diverge also in this case.

The second possible problem is that the depths of the $\rho_i$ might be bounded. Then there exists $k\ge 2$ for which there is a subsequence $\rho_{n_i}$ whose elements all have depth $k$ and we can choose $k$ minimal with that property. For each $i$, there must exists a precritical point of depth lower than $k$ on the path $[\rho_{n_i},\rho_{n_{i+1}}]$. Only finitely many of them can lie on the critical path, so infinitely many of them lie on a branch $[\tau_{n_i}, \rho_{n_i}]$, contradicting minimality of $k$. With these results, the proof for the case $d<\infty$ applies also in this case.
\end{proof}

\begin{example}[Branch point with infinite valency]
\label{Ex:InfiniteValency}
For $d=\infty$ and every kneading sequence in which infinitely many symbols occur, the critical point is always a branch point with infinite valency.\end{example}

Our final example shows that branch points can be dense on the Hubbard tree; 
for $d=\infty$, these branch points can all have infinite valency.

\begin{example}[Branch points are dense]
For all $d\in\{2,3,\dots,\infty\}$ and every kneading sequence $\nu$ such that every finite word over $A_d$ occurs somewhere within the sequence, all precritical points of $X_\nu$ lie in $\Hub(\nu)$ and have $d$ branches (except the critical value).
\end{example}

For each precritical point $w\star\nu$,  consider two postcritical points that start with $w\mathtt e$ and $w\mathtt{e'}$ for distinct $\mathtt{e},\mathtt{e'}\in A_d$. Then the precritical point must lie on the unique path between them. 


\end{document}